\documentclass[11pt]{article}

\usepackage{amsmath,amsthm,verbatim,amssymb,amsfonts,amscd, graphicx}
\usepackage{graphics}
\usepackage{authblk}
\usepackage{upgreek}
\usepackage{amsrefs,hyperref}
\theoremstyle{plain}
\newtheorem{theorem}{Theorem}[section]

\newtheorem{corollary}[theorem]{Corollary}
\newtheorem{definition}[theorem]{Definition}

\newtheorem{lemma}[theorem]{Lemma}
\newtheorem{proposition}[theorem]{Proposition}

\theoremstyle{remark}
\newtheorem{remark}[theorem]{Remark}

\numberwithin{equation}{section}
\newcommand{\diff}{\mathop{}\!\mathrm{d}}
\DeclareMathOperator*{\Hess}{Hess}
\DeclareMathOperator*{\tr}{tr}
\DeclareMathOperator*{\divr}{div}

\DeclareMathOperator\spn{span}

\title{Finite-time blowup for smooth solutions of the Navier--Stokes equations on the whole space with linear growth at infinity}
\author[1]{Evan Miller}
\affil[1]{McMaster University, 
Department of Mathematics and Statistics

emiller@msri.org}

\begin{document}
\maketitle

\begin{abstract}
In this paper we consider smooth solutions of the Navier--Stokes equations with a linear dependence on the spatial variable. We reduce the evolution of these solutions to a matrix ODE, and show that there are such solutions that blowup in finite-time. Note that because these solutions have linear growth at infinity, this blowup is not a counterexample disproving the global regularity of strong solutions of the Navier--Stokes equations, as strong solutions must have sufficient decay at infinity. This paper does not resolve the Millennium Problem. Nonetheless, these solutions do exhibit several properties that are closely related to the problem of blowup for strong solutions of Navier--Stokes equations, including the presence of unbounded planar stretching, and the alignment of the vorticity with the middle eigenvector of the strain matrix.
\end{abstract}

\section{Introduction}

The Navier--Stokes equations are the fundamental equations of fluid mechanics; the incompressible Navier--Stokes equations are given by
\begin{align} \label{NS}
    \partial_t u-\nu \Delta u+(u\cdot\nabla)u+\nabla p&=0 \\
    \nabla \cdot u &=0, \label{DivFree}
\end{align}
where $u$ is the velocity, $p$ is the pressure, and $\nu>0$ is the viscosity.
The equation \eqref{NS} expresses Newton's second law in the context of the flow of water.
$\partial_t u+(u\cdot\nabla)u$ is the acceleration in the Lagrangian frame, while $-\nabla p$ is the force due to the pressure, and $\nu\Delta u$ is the force from the internal friction of the water, which gives rise to viscosity, and to the viscous dissipation of energy.
The equation \eqref{DivFree} expresses the conservation of mass for an incompressible fluid, and is therefore often referred to as the incompressibility condition.

In this paper we will consider a class of smooth solutions to the Navier--Stokes equation with linear growth at infinity. These solutions will have a linear spatial dependence at each time, and will have the from
$u(x,t)=M(t)x$. We will show that this class of solutions exhibits finite-time blowup, although it is important to note at the outset that this does not resolve the Navier--Stokes regularity problem, as these solutions do not decay at infinity, and hence have infinite energy and are not in any Lebesgue or Sobolev spaces.

\begin{definition} \label{MatrixDef}
For $M\in C^1\left([0,T_{max});\mathbb{R}^{3\times 3}\right)$, we will say that $u(x,t)=M(t)x,$ is a spatially linear solution of the Navier--Stokes equation if
\begin{align} \label{MatrixNS}
    \partial_t M +M^2-\frac{1}{3}\tr\left(M^2\right)I_3&=0, \\
    \tr(M)&=0.
\end{align}
If $M$ satisfies \eqref{MatrixNS}, then we will say that $M$ is a solution of the matrix Navier--Stokes equation.
\end{definition}

This definition may appear somewhat arbitrary on the surface, however we can show that if $u$ satisfies Definition \ref{MatrixDef}, then $u$ is a classical solution of the Navier--Stokes equation.

\begin{proposition}
If $u\in C^1\left([0,T_{max});C^\infty\right)$ is a spatially linear solution of the Navier--Stokes equation, then $u$ is a classical solution of the Navier--Stokes equation with pressure
\begin{equation}
    p(x,t)=-\frac{1}{6}\tr\left(M(t)^2\right)|x|^2.
\end{equation}
\end{proposition}

\begin{proof}
$u(x,t)=M(t)x$ is a spatially linear solution of the Navier--Stokes equation, and so we know that
\begin{align}
    \partial_t M +M^2-\frac{1}{3}\tr\left(M^2\right)I_3&=0, \\
    \tr(M)&=0.
\end{align}
Next we observe that 
\begin{equation}
    \nabla u(x,t)=M^{tr}(t).
\end{equation}
Therefore we can conclude that
\begin{align}
    \nabla\cdot u
    &=\tr(\nabla u)\\
    &=\tr\left(M^{tr}\right) \\
    &=0,
\end{align}
and the divergence free condition is satisfied.
Next we observe that 
\begin{align}
    (u\cdot\nabla)u 
    &=
    (\nabla u)^{tr}u \\
    &=
    M(t)^2 x,
\end{align}
and that
\begin{equation}
    \nabla p(x,t)=-\frac{1}{3}\tr\left(M(t)^2\right)x.
\end{equation}
We can also see that 
\begin{equation}
    \partial_t u(x,t)=\partial_t M(t)x,
\end{equation}
and that 
\begin{equation}
    -\Delta u=0.
\end{equation}
Therefore we may conclude that
\begin{align}
    \partial_t u-\nu \Delta u+(u\cdot\nabla u)+\nabla p
    &=
    \left(\partial_t M+M^2-\frac{1}{3}\tr\left(M^2\right)I_3
    \right)x \\
    &=0.
\end{align}
This completes the proof.
\end{proof}

\begin{remark}
Note that while $u$ is a classical solution of the Navier--Stokes equations, it is not a Leray weak solution \cite{Leray} or a mild solution \cite{KatoFujita}. This is because the growth at infinity means that $u$ is not in the energy space $L^2$, or in any scale critical or subcritical spaces---$L^3, H^1, \dot{H}^1$, etc. This means that while $u$ is a classical solution of the Navier--Stokes equation, it is not really a strong solution.
We will also note that because $u$ is a linear function in the spatial variable, it must be harmonic with $-\Delta u=0$. This in particular means that $u$ is a classical solution of the Euler equation as well as the Navier--Stokes equation with arbitrary $\nu>0$.
\end{remark}

The structure of general matrices is not all that nice to work with so we will introduce equations for the evolution of the symmetric and anti-symmetric parts of $M$---with the anti-symmetric part represented by the vorticity vector $\omega$---which have additional structure that we can exploit. The fact that the symmetric matrices are diagonalizable will be a particularly crucial fact in our analysis.

\begin{definition} \label{StrainVortPairDef}
    For a strain $S\in C^1(\left[0,T_{max};
    \mathcal{S}^{3\times 3}\right), \tr(S)=0,$ 
    and a vorticity $\omega \in 
    C^1(\left[0,T_{max};\mathbb{R}^{3}\right)$,
    we will say that the pair $(S,\omega)$ is a solution of the strain-vorticity pair equation if
    \begin{align}
        \partial_t S+S^2-\frac{1}{3}|S|^2
        +\frac{1}{4}\omega\otimes\omega-\frac{1}{12}|\omega|^2 I_3
        &=0 \\
        \partial_t \omega-S\omega&=0.
    \end{align}
\end{definition}

\begin{definition}
    For $M\in \mathbb{R}^{3 \times 3}, \tr(M)=0$ take the associated strain and vorticity to be determined by the symmetric and anti-symmetric parts respectively. The associated strain $S$ is given by
    \begin{equation}
        S=\frac{1}{2}\left(M+M^{tr}\right),
    \end{equation}
    and the associated vorticity $\omega$ is given by
    \begin{equation}
        \left( \begin{array}{ccc}
             0& \omega_3 & -\omega_2  \\
             -\omega_3& 0 & \omega_1 \\
             \omega_2 & -\omega_1 &0
        \end{array}
        \right)
        = M^{tr}-M.
    \end{equation}
\end{definition}

We will show that a pair $(S,\omega)$ is a solution of the strain-vorticity pair equation if and only if the associated $M$ is a solution to the matrix Navier--Stokes equation.

\begin{proposition} \label{StrainVortEquivalence}
$M\in C^1\left([0,T_{max});\mathbb{R}^{3\times 3}\right), \tr\left(M\right)=0$ is solution of the matrix Navier--Stokes equation if and only if the associated strain-vorticity pair $(S,\omega)$ is a solution to the strain-vorticity pair equation.
\end{proposition}

\begin{proof}
Suppose $M\in C^1\left([0,T_{max});\mathbb{R}^{3\times 3}\right), \tr\left(M\right)=0$ is solution of the matrix Navier--Stokes equation.
We will begin by letting
\begin{align}
    S&=\frac{1}{2}\left(M+M^{tr}\right) \\
    A&=\frac{1}{2}\left(-M+M^{tr}\right),
\end{align}
where
\begin{equation}
    A=\frac{1}{2}
    \left( \begin{array}{ccc}
             0& \omega_3 & -\omega_2  \\
             -\omega_3& 0 & \omega_1 \\
             \omega_2 & -\omega_1 &0
        \end{array}
        \right).
\end{equation}
First we observe that $M=S-A$, 
and therefore we may observe that
\begin{align}
    |M|^2&=|S|^2+|A|^2 \\
    &= |S|^2+\frac{1}{2}|\omega|^2,
\end{align}
and that likewise
\begin{align}
    M^2
    &=
    (S-A)^2 \\
    &=
    S^2+A^2 -SA- AS,
\end{align}
and
\begin{equation}
    A^2=\frac{1}{4}\omega\otimes\omega
    -\frac{1}{4}|\omega|^2I_3.
\end{equation}
This implies that
\begin{equation}
    \tr\left(M^2\right)=|S|^2-\frac{1}{2}|\omega|^2.
\end{equation}
Therefore we can see that
\begin{equation}
    M^2-\frac{1}{3}\tr\left(M^2\right)I_3
    =
    S^2-\frac{1}{3}|S|^2I_3
    +\frac{1}{4}\omega\otimes\omega
    -\frac{1}{12}|\omega|^2 I_3
    -SA-AS.
\end{equation}
Using this we can see that
\begin{equation}
    \partial_t M+M^2
    -\frac{1}{3}\tr\left(M^2\right) I_3=0
\end{equation}
if and only if
\begin{equation} \label{M-writtenout}
    \partial_t S-\partial_t A
    +S^2-\frac{1}{3}|S|^2I_3
    +\frac{1}{4}\omega\otimes\omega
    -\frac{1}{12}|\omega|^2 I_3
    -SA-AS=0
\end{equation}
A matrix is only zero if both its symmetric and anti-symmetric parts are zero, so taking the symmetric and anti-symmetric parts of \eqref{M-writtenout} we find that
\begin{equation}
    \partial_t M+M^2
    -\frac{1}{3}\tr\left(M^2\right) I_3=0
\end{equation}
if and only if
\begin{align}
    \partial_t S+S^2-\frac{1}{3}|S|^2
    +\frac{1}{4}\omega\otimes\omega-\frac{1}{12}|\omega|^2 I_3
    &=0 \\
    \partial_t A+SA+AS&=0.
\end{align}

It now remains only to show that 
\begin{equation}
    \partial_t A+SA+AS=0
\end{equation}
if and only if
\begin{equation}
    \partial_t\omega-S\omega=0.
\end{equation}
Using the condition that $\tr(S)=0,$
it is an elementary linear algebra computation to observe that
\begin{equation}
    SA+AS=
    -\frac{1}{2}
    \left( \begin{array}{ccc}
             0& (S\omega)_3 & -(S\omega)_2  \\
             -(S\omega)& 0 & (S\omega)_1 \\
             (S\omega)_2 & -(S\omega)_1 &0
        \end{array}
        \right),
\end{equation}
and this completes the proof.

\end{proof}

\begin{remark}
In addition to the velocity, two other crucially important quantities in the Navier--Stokes and Euler equations are the strain and the vorticity. The strain is the symmetric part of the gradient, and is given by
$S_{ij}=\frac{1}{2}\left(\partial_i u_j+\partial_j u_i\right)$.
The vorticity is the curl of the velocity, $\omega=\nabla \times u,$ and is a vector representation of the anti-symmetric part of the gradient,
$A_{ij}=\frac{1}{2}\left(\partial_i u_j+\partial_j u_i\right),$ with
\begin{equation}
        A=\frac{1}{2}
    \left( \begin{array}{ccc}
             0& \omega_3 & -\omega_2  \\
             -\omega_3& 0 & \omega_1 \\
             \omega_2 & -\omega_1 &0
        \end{array}
        \right).
\end{equation}
Physically speaking, the vorticity represents the rotation induced by the fluid flow at a point, while the strain represents the deformation induced by the flow.
We will note that this is consistent with how we have defined $S$ and $\omega$ in the proof above in a way that is particular to spatially linear solutions of the Navier--Stokes equation, because when $u(x)=Mx,$ then
\begin{equation}
    \nabla u=M^{tr},
\end{equation}
and so it is consistent with the general definition that
$S=\frac{1}{2}\left(M+M^{tr}\right)$ and that
$A=\frac{1}{2}\left(-M+M^{tr}\right)$.
Note that the consistency of the definition of $\omega$ then also follows from the consistency of $A$. We also note here that is is straightforward to express $u$ in terms of $S$ and $\omega$ by
\begin{equation}
    u(x,t)=S(t)x+\frac{1}{2}\omega(t)\times x.
\end{equation}
\end{remark}

\begin{remark}
The evolution equation for the strain is given by
\begin{equation} \label{ClassicStrain}
    \partial_t S-\nu\Delta S+(u\cdot\nabla)S+S^2
    +\frac{1}{4}\omega\otimes\omega-\frac{1}{4}|\omega|^2I_3
    +\Hess(p)=0,
\end{equation}
and evolution equation for the vorticity is given by
\begin{equation} \label{ClassicVort}
    \partial_t\omega -\nu\Delta\omega 
    +(u\cdot\nabla)\omega-S\omega=0.
\end{equation}
Noting that vorticity and strain are constant in space for spatially linear solutions of the Navier--Stokes equation, it is clear that there is no advection or viscous dissipation of the strain or vorticity as 
$-\Delta S, -\Delta \omega, (u\cdot\nabla)S, (u\cdot\nabla)\omega=0$.
Further observing that 
\begin{align}
    p
    &=
    -\frac{1}{6}\tr\left(M^2\right)|x|^2 \\
    &=
    \left(-\frac{1}{6}|S|^2+\frac{1}{12}|\omega|^2\right)|x|^2 \\,
\end{align}
we find that 
\begin{equation}
    \Hess(p)= 
    \left(-\frac{1}{3}|S|^2+\frac{1}{6}|\omega|^2\right) I_3.
\end{equation}
Therefore we can see that for the case of spatially linear solutions of the Navier--Stokes equation, solutions of the strain-vorticity pair equation equation satisfy the evolution equations \eqref{ClassicStrain} and \eqref{ClassicVort} with the choice of pressure indicated.
\end{remark}

\section{Statements of the main results}

In this section, we will state the main results of the paper and discuss their relationship to the literature for the general case of strong solutions of the Navier--Stokes and Euler equations, for which the spatially linear solutions are a toy model. We will begin by considering a class of blowup solutions that preserves its geometric structure with the dynamics.

\begin{theorem} \label{BlowupFamIntro}
Suppose we have initial data
\begin{equation}
    S^0=
    \left(
    \begin{array}{ccc}
         -(r+1)\lambda_0& 0 & 0  \\
          0& r \lambda_0 &0 \\
          0 &0 &\lambda_0 
    \end{array}
    \right)
\end{equation}
and 
\begin{equation}
    \omega^0= 
    \left(
    \begin{array}{c}
         0  \\
         2k\lambda_0 \\
         0
    \end{array}
    \right),
\end{equation}
where $\lambda_0>0, -\frac{1}{2}\leq r\leq 1,$ and
$k^2=1+r-2r^2$.
If we let
\begin{equation}
    \partial_t\lambda=r\lambda^2,
\end{equation}
then the pair $(S,\omega)$ given by 
\begin{equation}
    S(t)=
    \left(
    \begin{array}{ccc}
         -(r+1)\lambda(t)& 0 & 0  \\
          0& r \lambda(t) &0 \\
          0 &0 &\lambda(t)
    \end{array}
    \right)
\end{equation}
and
\begin{equation}
        \omega(t)= 
    \left(
    \begin{array}{c}
         0  \\
         2k\lambda(t) \\
         0
    \end{array}
    \right)
\end{equation}
is a solution to the strain-vorticity pair equation
with initial data $\left(S^0,\omega^0\right)$.

The general form for this solution is given by
\begin{equation}
    \lambda(t)=\frac{\lambda_0}{1-r\lambda_0 t}.
\end{equation}
When $r>0$ this solution blows up in finite time,
$T_{max}= \frac{1}{r\lambda_0}$.
When $r=0$, this is a stationary solution, and hence $T_{max}=+\infty.$
When $r<0, T_{max}=+\infty$ and this solution decays at infinity proportional to $\frac{1}{t}$.
\end{theorem}

\begin{remark}
Applying Proposition \ref{StrainVortEquivalence},
we can see that the solutions of the spatially linear Navier--Stokes equation corresponding to the solutions of the strain-vorticity pair equation in Theorem \ref{BlowupFam}
are given by
\begin{equation}
    u(x,t)=M(t)x,
\end{equation}
where
\begin{equation}
    M(t)=\frac{\lambda_0}{1-r\lambda_0 t}
    \left(\begin{array}{ccc}
         -(1+r)& 0 & k  \\
         0& r & 0 \\
         -k &0 & 1
    \end{array}
    \right),
\end{equation}
where we again have $-\frac{1}{2}\leq r\leq 1$ and
$k^2=1+r-2r^2$.

\end{remark}

\begin{remark} \label{VortAlign}
Note that because $-\frac{1}{2}\leq r\leq 1$, we can see that
$-(r+1)\leq r \leq 1$. This means that if we let
\begin{align}
    \lambda_1 &=-(r+1)\lambda \\
    \lambda_2 &= r\lambda \\
    \lambda_3&= \lambda,
\end{align}
then we will have 
$\lambda_1\leq \lambda_2\leq \lambda_3$.
We have parametrized the entire system in terms of $\lambda_3$, because $\lambda_3$ is always positive unless $S$ is identically zero; however, when $\lambda_2>0$, we can just as easily parametrize the system by $\lambda_2$. If we do that, we get the following ODE governing blowup:
\begin{align}
    \partial_t\lambda_2
    &=
    \partial_t r\lambda \\
    &= r^2\lambda^2 \\
    &=\lambda_2^2.
\end{align}
This means that there is only blowup when $\lambda_2>0$, and the size of $\lambda_2$ completely dictates the rate of blowup with 
\begin{equation}
    T_{max}=\frac{1}{\lambda_2(0)}.
\end{equation}

This is consistent with earlier work on mild solutions of the Navier--Stokes equation proving scale critical regularity criteria on $\lambda_2^+$, the positive part of the middle eigenvalue of the strain \cites{NeuPen1,NeuPen2,MillerStrain}. In these work scale critical norms applied to $\lambda_2^+$ controlled regularity, and correspondingly in this simplified model, there is blowup when $\lambda_2(0)$ is positive, and the size of $\lambda_2(0)$ alone dictates the rate of blowup.
We should note that physically speaking two positive eigenvalues and one negative eigenvalue corresponds to planar stretching, and axial compression, with the axial compression being stronger in magnitude, because $\lambda_1=-\lambda_2-\lambda_3$.

It is also worth noting that in the family of solutions we have described, the vorticity aligns entirely with the eigenvector corresponding middle eigenvalue of the strain matrix. This is consistent with previous work suggesting a tendency of the vorticity to align with the eigenvector corresponding to the middle eigenvalue \cite{VortAlignment}, and suggests that there may be a deep relationship between the alignment of the vorticity with the middle eigenvector and the possibility of unbounded planar stretching as a mechanism for blowup.

Finally, we should mention that these blowup solutions are also consistent with work on the pressure by Seregin and S\v{v}er\'ak \cite{SereginSverakPressure}. Seregin and S\v{v}er\'ak showed that in order for a solution of the Navier--Stokes equation to blowup in finite-time, then $p$ must become unbounded below and $p+\frac{1}{2}|u|^2$ must become unbounded above. We will discuss this further in section \ref{Lagrangian}.
\end{remark}

\begin{theorem} \label{MainTheorem}
Suppose a strain $S\in C^1(\left[0,T_{max};
\mathcal{S}^{3\times 3}\right), \tr(S)=0$, 
and a vorticity $\omega \in 
C^1(\left[0,T_{max};\mathbb{R}^{3}\right)$,
is a solution of the strain-vorticity pair equation 
with initial data
\begin{equation}
    S^0=
    \left(
    \begin{array}{ccc}
         -(r_0+1)\lambda_0& 0 & 0  \\
          0& r_0 \lambda_0 &0 \\
          0 &0 &\lambda_0 
    \end{array}
    \right)
\end{equation}
and 
\begin{equation}
    \omega^0= 
    \left(
    \begin{array}{c}
         0  \\
         2k_0\lambda_0 \\
         0
    \end{array}
    \right),
\end{equation}
with $\lambda_0>0$ and $k_0^2\neq 1+r_0-2r_0^2$.
Then the pair $S,\omega\in C^1\left([0,T_{max})\right)$
is a solution of the strain-vorticity pair equation with initial data $S^0,\omega^0$ where
the strain is given by
\begin{equation}
    S(t)=
    \left(
    \begin{array}{ccc}
         -(r(t)+1)\lambda(t)& 0 & 0  \\
          0& r(t) \lambda(t) &0 \\
          0 &0 &\lambda(t) 
    \end{array}
    \right),
\end{equation}
and a vorticity given by
\begin{equation}
    \omega(t)= 
    \left(
    \begin{array}{c}
         0  \\
         2k(t)\lambda(t) \\
         0
    \end{array}
    \right),
\end{equation}
and $r,k,\lambda\in C^1\left([0,T_{max})\right)$ satisfy
\begin{align}
    \partial_t\lambda &=
    \frac{1}{3}\left(-1+2r+2r^2+k^2\right) \lambda^2 \\
    \partial_t r &=
    \frac{1}{3}\lambda\left(2+3r-3r^2-2r^3-(r+2)k^2\right) \\
    \partial_t k &=\frac{1}{3}\lambda (1+r-2r^2-k^2)k.
\end{align}

For all $r_0\neq -2$, let 
\begin{equation}
    m_0=\frac{k_0}{r_0+2},
\end{equation}
and if $-\frac{1}{2}<m_0<\frac{1}{2},$ then let
\begin{equation}
    r_\infty=\frac{1-4m_0^2+3\sqrt{1-4m_0^2}}{2m_0^2+4},
\end{equation}
and let
\begin{equation}
    k_\infty= m_0
    \left(\frac{9+3\sqrt{1-4m_0^2}}{2m_0^2+4}\right).
\end{equation}
Furthermore, we will let
\begin{equation}
    g(t)=1+r(t)-2r(t)^2-k(t)^2
\end{equation}
The properties of the solutions for this initial data are as follows:
\begin{enumerate}
    \item 
If $r_0>0$ and
$1+r_0-2r_0^2<k_0^2<\frac{1}{4}(r_0+2)^2,$
then $T_{max}<+\infty$ and 
\begin{align*}
    \lim_{t\to T_{max}} r(t)&=r_\infty, \\
    \lim_{t\to T_{max}}k(t)&=k_\infty, \\
    \lim_{t\to T_{max}}g(t)&=0, \\
    \lim_{t\to T_{max}}\lambda(t)&=+\infty.
\end{align*}
\item
If $-\frac{1}{2}<r_0< 1$ and 
$k_0^2<1+r_0-2r_0^2$,
then $T_{max}<+\infty,$ and
\begin{align*}
    \lim_{t\to T_{max}} r(t)&=r_\infty, \\
    \lim_{t\to T_{max}}k(t)&=k_\infty, \\
    \lim_{t\to T_{max}}g(t)&=0, \\
    \lim_{t\to T_{max}}\lambda(t)&=+\infty.
\end{align*}
\item 
If $r_0>0$ and $k_0^2=\frac{1}{4}(r_0+2)^2$,
then $T_{max}=+\infty$, and
\begin{align*}
    \lim_{t\to +\infty} r(t)&=0, \\
    \lim_{t\to +\infty}k(t)&=1, \\
    \lim_{t\to +\infty}g(t)&=0.
\end{align*}
There is, however, blowup at infinity with
\begin{equation*}
    \lim_{t\to +\infty}\lambda(t)=+\infty.
\end{equation*}
\item
If $r_0>0, k_0^2>\frac{1}{4}(r_0+2)^2$,
then $T_{max}<+\infty,$ and
\begin{align*}
    \lim_{t\to T_{max}} r(t)&=-2, \\
    \lim_{t\to T_{max}}k(t)&=0, \\
    \lim_{t\to T_{max}}\lambda(t)&=+\infty.
\end{align*}
\item
If $-\frac{1}{2}\leq r_0\leq 0$ and $k_0^2>1+r_0-2r_0^2$,
then $T_{max}<+\infty,$ and
\begin{align*}
    \lim_{t\to T_{max}} r(t)&=-2, \\
    \lim_{t\to T_{max}}k(t)&=0, \\
    \lim_{t\to T_{max}}\lambda(t)&=+\infty.
\end{align*}

\item
If $r_0< -\frac{1}{2}$,
then $T_{max}<+\infty,$ and
\begin{align*}
    \lim_{t\to T_{max}} r(t)&=-2, \\
    \lim_{t\to T_{max}}k(t)&=0, \\
    \lim_{t\to T_{max}}\lambda(t)&=+\infty.
\end{align*}
\end{enumerate}
\end{theorem}

\begin{remark} \label{StabilityOfBlowup}
We will note that it is easy to compute that 
\begin{equation}
    k_\infty^2=1+r_\infty-2r_\infty^2,
\end{equation}
so the family of blowup solutions in Theorem \ref{BlowupFam}
where $0<r\leq 1$ and $k^2=1+r-2r^2$
is stable under small perturbations of the parameters $r,k$.
Furthermore, the set of initial data covered in conditions $1$ and $2$ of Theorem \ref{MainTheorem} approach the family of blowup solutions in Theorem \ref{BlowupFamIntro}
as $t\to T_{max}$. 

The family of solutions with global regularity and decay at infinity in Theorem \ref{BlowupFam},
the case where $-\frac{1}{2}<r<0$,
are likewise unstable with respect small perturbations of the parameters $r,k$. The stationary solution, the case where $r=0$, is also unstable. This solution can, however, be seen as an attractor for the borderline case where $r>0, k^2=\frac{1}{4}(r+2)^2$. In this case, there is a global smooth solution. This solution blows up at infinity, so it does not approach any specific stationary solution, but it blows up quite slowly---and hence at infinity and not in finite-time---precisely because $k\to 1, r\to 0$ as $t\to +\infty$, which means the parameters are approaching the stationary solution.
\end{remark}

\begin{remark}
In \cite{MillerStrain}, the author introduced the matrix ODE toy model for the strain equation
\begin{equation} \label{StrainODE}
    \partial_t S+S^2-\frac{1}{3}|S|^2I_3=0.
\end{equation}
It is easy to check that this equation is a special case of the strain-vorticity pair equation where vorticity is zero and also, equivalently, a special case of the matrix Navier--Stokes equation restricted to symmetric matrices.
The equation \eqref{StrainODE} was introduced as a toy model ODE for the strain equation, and it was not clear at this time that solutions of this equation yielded classical solutions to the Navier--Stokes equation with linear growth at infinity.

The author showed in \cite{MillerStrain}, that for initial data
\begin{equation}
    S^0=
    \left(
    \begin{array}{ccc}
         -(r_0+1)\lambda_0& 0 & 0  \\
          0& r_0 \lambda_0 &0 \\
          0 &0 &\lambda_0 
    \end{array}
    \right),
\end{equation}
if $\lambda_0>0$ and $-\frac{1}{2}<r_0\leq 1$, then
$T_{max}<+\infty$
and
\begin{equation}
    \lim_{t \to T_{max}} r(t)=1.
\end{equation}
This results is in fact a special case of Theorem \ref{MainTheorem}, condition 2, when we have $k_0=0$. Note that if $k_0=0$, then $m=0,$ and hence $k_\infty=0$, and $r_\infty=1$.
\end{remark}

\begin{remark}
There is a growing body of research suggesting that the self-amplification of strain is the dominant factor in the growth of enstrophy \cite{MillerStrainToyModel} 
and in the turbulent energy cascade \cites{Tsinober,CarboneBragg}, 
and is in fact more important than vortex stretching which had long been assumed to play the main role.
For mild solutions of the Navier--Stokes equation in $H^1$, the nonlinearity for the growth of enstrophy can be formulated equivalently in terms of the self-amplification of strain, given by $-4\int\det(S)$ and the vortex stretching, given by $\left<S,\omega\otimes\omega\right>$. While these formulations are equivalent however, the self-amplification of strain is more useful as it is entirely local, whereas vortex stretching requires a Riesz-type transform to get $S$ from $\omega$. See \cites{MillerStrain} for more details.

For this reason, it is interesting that the blowup solutions of the strain-vorticity pair equation remain stable in the presence of vorticity. This is particularly important given that any nonzero, finite-energy fluid cannot be irrotational, so any finite-time blowup for the Navier--Stokes equation in $H^1$ would have to be in the presence of vorticity. The interaction of the strain and vorticity may dampen the rate of quadratic blowup, noting in particular for the family of blowup solutions in Theorem \ref{BlowupFam} that $r<1$ when $k>0,$ but there is no depletion of nonlinearity, as we still have
\begin{equation}
    \partial_t\lambda=r\lambda^2.
\end{equation}
This suggests that any depletion of nonlinearity cannot come from the interaction of strain and vorticity and must come from other effects, particularly advection.
\end{remark}

\begin{remark}
This work does not, of course, resolve the Millennium Problem for the Navier--Stokes equations. While the blowup solutions we have exhibited are smooth, the fact that they are unbounded at infinity, and hence have infinite energy and are not in any Lebesgue or Sobolev space means that these blowup solutions can really only be seen as a toy model for Navier--Stokes. There is no natural way to extend the blowup for these solutions to blowup for mild solutions, where there must be decay at infinity. This is true in particular because the direction of the vorticity is constant for the class of solutions dealt with in Theorems \ref{BlowupFamIntro} and \ref{MainTheorem}, and the classical result of Constantin and Fefferman \cite{ConstantinFefferman} shows that the vorticity direction must vary rapidly in order for there to be blowup for mild solutions of the Navier--Stokes equation. Any method of trying to impose decay on the blowup solutions derived in this paper will necessarily run into enormous technical difficulties, as the non-locality of the equations can no longer be ignored.

The advantage of the toy model in this paper when compared with other toy models, is that the solutions are actually classical solutions of the Navier--Stokes equation. The general approach in proving finite-time blowup for model equation for Navier--Stokes is to relax the equation, but preserve the key function spaces
\cites{MillerStrainToyModel,TaoModel,MontgomerySmith,GallagherPaicu}. 
Here we have taken the opposite approach, relaxing the function space requirements, but keeping the equations intact. Insofar as the goal of Navier--Stokes analysis is to gain further mathematical insight into the physics of incompressible flow, this has certain benefits. Even though the equation reduces to an ODE, certain fundamental features of actual turbulent flows and the possible blowup of strong solutions are captured by this model, 
as we described in Remark \ref{VortAlign}. 
\end{remark}

\begin{remark}
Theorem \ref{MainTheorem} shows that the set of blowup solutions from Theorem \ref{BlowupFamIntro} are stable, while the the set of global solutions is unstable.
This is highly significant because,
while the solutions to the Navier--Stokes equations described in Theorem \ref{MainTheorem} may have issues globally in space due to their linear growth at infinity, these solutions are perfectly well behaved locally.
The family of blowup solutions in Theorem \ref{BlowupFamIntro}, and the solutions in Theorem \ref{MainTheorem} that approach the this family of blowup solutions as $t\to T_{max}$, do not exhibit any depletion of nonlinearity, and have fully quadratic blowup. This suggests that any depletion of nonlinearity for the Navier--Stokes equations relies strongly on nonlocal effects.

For strong solutions of the Navier--Stokes equation, we can invert either the strain or the vorticity to get the velocity, and therefore the strain and vorticity entirely determine each other in terms of zero order pseudo-differential operators \cite{MillerStrain}  with
\begin{align}
    S&=\nabla_{sym} \nabla \times
    (-\Delta)^{-1}\omega \\
    \omega &= -2 \nabla \times \divr
    (-\Delta)^{-1}S.
\end{align}
Because the vorticity and strain are just the symmetric and anti-symmetric parts of a generic trace free matrix in this set up, the matrix $M$ and the velocity $u=Mx$ cannot be recovered from just $S$ or $\omega$. The mutual dependence of $S$ and $\omega$ is nonlocal, and so the growth at infinity allows $S$ and $\omega$ to uncouple.
\end{remark}

\begin{remark}
There is a large body of research showing that advection plays a regularizing role in the equations of fluid dynamics and related models
\cites{DiegoCordoba,ElgindiAdvection,Sverak1DModels,HouLei}.
We should note that because $S$ and $\omega$ are constant in space, there is no advection of either strain or vorticity, as we have
\begin{align}
    (u\cdot\nabla)S&=0 \\
    (u\cdot\nabla ) \omega &=0.
\end{align}
This means that advection cannot play a regularizing role for the spatially linear Navier--Stokes equation, and so finite-time blowup for this equation is consistent with earlier works on the regularizing role of advection. This work suggests that the quadratic nonlinearities in the strain and vorticity are ``trying" to form coherent coherent structures that would lead to finite-time blowup, but that, for the full Navier--Stokes problem with a finite-energy constraint and decay at infinity, the large velocities generated may advect away such structures before they can blowup the equation. If there is a depletion of nonlinearity this is the most likely source, as the strain-vorticity interaction does not seem to lead to depletion of nonlinearity.
\end{remark}

\begin{remark}
Spatially linear solutions of the Navier--Stokes equation are not unique in the class of classical solutions to the Navier--Stokes equation.
For example take 
\begin{equation}
    u^0(x)=\left(\begin{array}{c}
         -2x_1  \\
         x_2 \\
         x_3
    \end{array}\right).
\end{equation}
We can see from Theorem \ref{BlowupFam}, that the solution of the spatially linear Navier--Stokes equation with this initial data is given by
\begin{equation}
    u(x,t)= \frac{1}{1-t}
    \left(\begin{array}{c}
         -2x_1  \\
         x_2 \\
         x_3
    \end{array}\right),
\end{equation}
with the corresponding pressure given by
\begin{equation}
    p(x,t)=-\frac{1}{1-t}|x|^2.
\end{equation}
This is not however, the only classical solution of the Navier--Stokes equation with this initial data.
In fact, with a different choice of the pressure,
\begin{equation}
    p(x)=-2x_1^2-\frac{1}{2}x_2^2-\frac{1}{2}x_3^2,
\end{equation}
we have that
\begin{equation}
    u(x)=\left(\begin{array}{c}
         -2x_1  \\
         x_2 \\
         x_3
    \end{array}\right)
\end{equation}
is a stationary solution of the Navier--Stokes equation.
More generally if we take a pressure of the form
\begin{equation}
    p(x)=-2x_1^1-\frac{1}{2}x_2^2-\frac{1}{2}x_3^2
    -f(t)\left(-x_1^2+\frac{1}{2}x_2^2+\frac{1}{2}x_3^2
    \right),
\end{equation}
where $f$ is a generic smooth function of time,
and we let $\partial_t \lambda=f,$ with $\lambda(0)=1$, then 
\begin{equation}
    u(x,t)=\lambda(t)
    \left(\begin{array}{c}
         -2x_1  \\
         x_2 \\
         x_3
    \end{array}\right)
\end{equation}
is a classical solution of the Navier--Stokes equation
with initial data $u^0$.
This makes it clear that classical solutions with linear growth at infinity are not only not unique, but can evolve essentially in any specified way, either growing, or shrinking, at any specified rate.

The source of this nonuniqueness is clearly the pressure. For $H^1$ solutions of the Navier--Stokes equations, the pressure is uniquely determined by the velocity by inverting the Laplacian with
\begin{align}
    -\Delta p
    &=
    \sum_{i,j=1}^3 \frac{\partial u_j}{\partial x_i}
    \frac{\partial u_i}{\partial x_j} \\
    &=
    |S|^2-\frac{1}{2}|\omega|^2.
\end{align}
For solutions with linear growth at infinity, we lose this unique determination of the pressure by the velocity, because we can no longer find a unique solution to Poisson's equation.
Nonetheless, spatially linear solutions of the Navier--Stokes equation are unique, because in requiring that $u(x,t)=M(t)x$ where
\begin{equation}
    \partial_t M+M^2-\frac{1}{3}\tr\left(M^2\right)=0,
\end{equation}
we have fixed pressure to be 
\begin{equation}
    p(x,t)=-\frac{1}{6}\tr\left(M(t)^2\right)|x|^2.
\end{equation}
While this is not the only choice of pressure, because our solution has a linear dependence on the spatial variable, the pressure can be seen as canonical because it is a radial function.
It is clear that the pressure function must be quadratic (note the above examples), because $\nabla p$ must be linear in order to preserve the linear structure of our equation.
Taking the pressure to be the quadratic, radial function that maintains the divergence free constraint is the natural canonical choice, and so in this sense, the solution given by Definition \ref{MatrixDef} 
for $u^0$ is the ``correct" solution.
\end{remark}

\begin{remark}
There are previous examples of infinite energy blowup solutions for the Euler equation \cite{GibbonInfiniteEnergy} and the Navier--Stokes equation \cites{BurgersVortex,MMPnonlocal}.
In particular, Maekawa, Miura, and Prange \cite{BurgersVortex} consider the evolution of the Burgers vortex a uniform strain of the form
\begin{equation}
        S(t)=
    \lambda(t)\left(
    \begin{array}{ccc}
         -1 & 0 & 0  \\
          0& -1  &0 \\
          0 &0 &2 
    \end{array}
    \right),
\end{equation}
with $\lambda(t)>0$,
which is quite interesting because this has precisely the opposite sign as the blowup solutions dealt with in this paper. For solutions of the form dealt with in Theorem \ref{MainTheorem}, finite-time blowup requires that $\lambda_2(t)>0$ for some $t>0$.
It is not immediately clear what in the structure of Burgers vortex allows solutions of this form, even with infinite energy, to blowup, but the difference is quite striking.
The more general form of linear strain with stretching in one direction and compression in two directions, given by
\begin{equation}
            S(t)=
    \lambda(t)\left(
    \begin{array}{ccc}
         -1 & 0 & 0  \\
          0& -m  &0 \\
          0 &0 &1+m 
    \end{array}
    \right),
\end{equation}
with $\lambda(t),m\geq 0$
was also considered in \cite{Beronov}, but this paper does not deal with finite-time blowup.
\end{remark}

\begin{remark}
One advantage of the present work compared with previous works on infinite energy blowup solutions to the Euler and Navier--Stokes equations is that the strain structure given by
    \begin{equation}
            S(t)=
    \lambda(t)\left(
    \begin{array}{ccc}
         -(1+r) & 0 & 0  \\
          0& r  &0 \\
          0 &0 &1 
    \end{array}
    \right),
\end{equation}
with $\lambda(t),r>0$, corresponds to a linear flow with a stagnation point at $x=0$. Physically this flow involves colliding jets moving along the $x$ axis, and stretching in the $yz$ plane, albeit at different rates along the $y$ and $z$ axes when $0<r<1$. 
Colliding jets have long been considered possible candidates for blowup for the Euler and Navier--Stokes equations, and recently played a major role in Elgindi's proof of finite-time blowup for $C^{1,\alpha}$ solutions of the Euler equation \cite{ElgindiBlowup}. One difficulty in such a program for proving finite-time blowup is that if the colliding jets are axisymetric and swirl-free---the case in this paper when $r=1,\omega=0,$---then this rules out blowup for the Navier--Stokes equation, and also for the Euler equation with sufficient regularity. 
The introduction of vorticity in the $y$ direction allows us to break the axisymmetry, because in this case we now have solutions in Theorem \ref{BlowupFamIntro}, where $0<r<1,$ which means the radial component of the velocity now depends on $\theta$. This means such solutions could be good candidates for the blowup of strong solutions of the Euler equations, if sufficient decay at infinity could be imposed.
\end{remark}

\begin{remark}
We will note that in Theorem \ref{MainTheorem}, we restrict ourselves to considering initial data where the vorticity aligns entirely with one of the eigenvectors of the strain. In general, the vorticity does not have to align completely with a particular eigenvector of the strain.
We in fact have a more general theorem, showing that the family of blowup solutions in Theorem \ref{BlowupFamIntro} is stable under generic perturbations, not merely under perturbations of the parameters in which the vorticity aligns entirely with one of the eigenvectors of the strain.
\end{remark}

\begin{theorem} \label{GeneralCase}
Suppose we have initial data 
$M^0\in\mathbb{R}^{3\times 3},
\tr\left(M^0\right)=0$,
with eigenvalues such that 
$\operatorname{Re}(\lambda_1) <
\operatorname{Re}(\lambda_2) \leq 
\operatorname{Re}(\lambda_3)$, and Jordan normal form
\begin{equation}
    M^0= Q J Q^{-1}.
\end{equation}
Then the solution of the matrix Navier--Stokes equation
$M\in C^1\left([0,T_{max};\mathbb{R}^{3\times 3}\right)$
blows up in finite-time $T_{max}<+\infty$,
and has the form
\begin{equation}
    M(t) = Q J(t) Q^{-1},
\end{equation}
where 
\begin{equation}
    J(t) \sim \frac{1}{T_{max}-t} \left(
    \begin{array}{ccc}
        -2 & 0 & 0  \\
         0& 1& 0 \\
         0 & 0 & 1
    \end{array} \right)
\end{equation}
as $t \to T_{max}$.
\end{theorem}

\begin{remark}
Theorem \ref{GeneralCase} in fact shows that the set of blowup solutions is stable even when the vorticity is not aligned with one of the eigenvectors of the strain rate, and is in fact a global attractor for all cases where
$\operatorname{Re}(\lambda_1) <
\operatorname{Re}(\lambda_2) \leq 
\operatorname{Re}(\lambda_3)$.
This is because if 
\begin{equation}
    M=\left(\begin{array}{ccc}
         -(1+r) & 0 & k  \\
         0 & r & 0 \\
         -k &0 & 1
    \end{array}\right),
\end{equation}
where $0<r\leq 1, k^2=1+r-2r^2$, 
then $M=QDQ^{-1}$,
where
\begin{equation}
    D=\left(\begin{array}{ccc}
         -2r &0 &0  \\
         0 & r &0 \\
         0 &0 & r
    \end{array}
    \right),
\end{equation}
and
\begin{equation}
    Q=\left(\begin{array}{ccc}
         \frac{1+2r}{\sqrt{k^2+(1+2r)^2}} &0 &\frac{k}{\sqrt{k^2+(1+2r)^2}}  \\
         0 & 1 &0 \\
         \frac{k}{\sqrt{k^2+(1+2r)^2}} &0 & \frac{1+2r}{\sqrt{k^2+(1+2r)^2}}
    \end{array}
    \right).
\end{equation}
This means that the asympotics of blowup for the broad class of initial data in Theorem \ref{GeneralCase} is in fact precisely the blowup solutions from Theorem \ref{BlowupFamIntro}.
\end{remark}

\begin{remark}
The matrix Navier--Stokes equation
\begin{equation}
    \partial_t M +M^2
    -\frac{1}{3} \tr\left(M^2\right)I_3=0,
\end{equation}
was previously considered by Cantwell in \cite{Cantwell}.
Finite-time blowup was observed in this paper, as well as the alignment of the vorticity with the middle eigenvector of the strain for a certain class of blowup solutions. The analysis in \cite{Cantwell} was quite different than that in the present paper. It was based on a second order linear differential equation involving Jacobian elliptic functions, rather than a first order nonlinear equation involving the Jordan normal form. The author would like to thank Michele Dolce for bringing the paper \cite{Cantwell} to his attention.
\end{remark}

Thus far we have considered the Eulerian formulation of the fluid mechanics, where we look at the fluid velocity at each point in time. The other approach that can be taken is the Lagrangian formulation of fluid mechanics, where we follow the evolution of particles, tracking where a given fluid parcel at $t=0$ has moved for $t>0$. 
In general, the Eulerian formulation of fluid mechanics is more conducive to analysis, but the Lagrangian framework is also important, because it is easier from an observation point of view to track the flow of fluid parcels (using dyes for instance), then to measure the velocity at specific points. This means that the Lagrangian formulation can be necessary from a data gathering point of view. There is also a connection between the Euler equation and the principle of least action involving optimal transport that requires the Lagrangian formulation \cites{Brenier}.

For the Lagrangian formulation, we consider where a particle that begins at $y_0$ will be at time $t$.
This is governed by the differential equation
\begin{equation}
    \partial_t y(t)=u(y(t),t)
\end{equation}
We can also consider the general Lagrangian flow map associated to the velocity vector field.
The flow map $Y(y_0,t)$ can be determined by the velocity field with the differential equation
\begin{equation} \label{LagrangeDiffEq}
    \partial_t Y(y_0,t)
    =u(Y(y_0,t),t),
\end{equation}
with the initital condition
\begin{equation}
    Y(y_0,0)=y_0.
\end{equation}
The map $Y(y_0,t)$ tells us where a particle that starts out at $y_0$ will be at time $t$.
We will note that, subject to certain regularity conditions, the flow map $y(y_0,t)$ and the velocity field $u(x,t)$ completely determine each other. Both fully characterize the flow and the fluid dynamical problem. There is, however, no straightforward evolution equation for $Y$ without having previously solved for $u$.
Because the flow is incompressible, the map $Y(\cdot,t)$ must be volume preserving.
It can be shown that this is implied by differential equation \eqref{LagrangeDiffEq} and the divergence free constraint $\nabla\cdot u=0$.

We can solve explicitly for the Lagrangian flow map for the family of blowup solutions described in Theorem \ref{BlowupFamIntro}. The conclusions are rather interesting physically.

\begin{theorem} \label{LagrangianIntro}
Let $Y(y_0,t)$ be the Lagrangian map associated with the velocity field
\begin{equation}
    u(x,t)=M(t)x,
\end{equation}
where for all $0\leq t<T_{max}=\frac{1}{r\lambda_0}$
\begin{equation}
    M(t)=\frac{\lambda_0}{1-r\lambda_0 t}
    \left(\begin{array}{ccc}
         -(1+r)& 0 & k  \\
         0& r & 0 \\
         -k &0 & 1
    \end{array}
    \right),
\end{equation}
where $0<r\leq 1, k^2=1+r-2r^2$.
That is let
\begin{equation}
    \partial_t Y(y_0,t)=u(Y(y_0,t),t)
\end{equation}
and
\begin{equation}
    Y(y_0,0)=y_0.
\end{equation}
Then for all $0\leq t<T_{max}$,
\begin{equation}
    Y(y_0,t)=QD(t)Q^{-1}y_0,
\end{equation}
where 
\begin{equation}
    D(t)=\left(\begin{array}{ccc}
         (1-r\lambda_0t)^2 &0 &0  \\
         0 & \frac{1}{1-r\lambda_0t} &0 \\
         0 &0 & \frac{1}{1-r\lambda_0t}
    \end{array}
    \right),
\end{equation}
and
\begin{equation}
    Q=\left(\begin{array}{ccc}
         \frac{1+2r}{\sqrt{k^2+(1+2r)^2}} &0 &\frac{k}{\sqrt{k^2+(1+2r)^2}}  \\
         0 & 1 &0 \\
         \frac{k}{\sqrt{k^2+(1+2r)^2}} &0 & \frac{1+2r}{\sqrt{k^2+(1+2r)^2}}
    \end{array}
    \right),
\end{equation}

This can be equivalently expressed as an ODE.
Suppose
\begin{equation}
    \partial_t y(t)=u(y(t),t),
\end{equation}
and $y_0=c_1 v_1+c_2v_2+c_3v_3$, where $v_1,v_2,v_3$ are the column vectors of $Q$.
Then for all $0\leq t<T_{max}$,
\begin{equation}
    y(t)=c_1 (1-r\lambda_0 t)^2v_1
    +c_2\frac{1}{1-r\lambda_0t}v_2
    +c_3 \frac{1}{1-r\lambda_0t} v_3.
\end{equation}
\end{theorem}

\begin{remark} \label{RotationRemarkIntro}
It may appear at first that there is no rotation in the Lagrangian frame, for the family of blowup solutions in Theorem \ref{BlowupFamIntro}, only equally strong stretching in the the direction of $v_2, v_3$, and compression in the direction of $v_1$. This would appear to contradict the both the presence of vorticity, and the fact that $\lambda_2<\lambda_3$ when $0<r<1$.
The rotation due to the vorticity comes into play when we consider the fact that the eigenbasis of the solutions $M(t)$ of the matrix Navier--Stokes equation is not an orthogonal basis when $0<r<1$, and consequently $k\neq 0$.
The effect of vorticity is to rotate the direction of greatest stretching and greatest compression about the $y$-axis, which is both the weaker direction of stretching and the direction of vorticity, so that they are no longer orthogonal. This means the rotation in the Lagrangian frame is built into the change of basis matrices $Q$ and $Q^{-1}$
in our expression
\begin{equation}
    Y(y_0,t)=QD(t)Q^{-1}y_0,
\end{equation}
precisely because the eigenbasis is not orthogonal, and consequently the matrix $Q$ is not orthogonal.
We will work through some examples in Section 5 showing that rotation does in fact occur in the Lagrangian frame, due to the presence of vorticity, for the family of blowup solutions from Theorem \ref{BlowupFamIntro} when $0<r<1$.
\end{remark}

\begin{remark}
The Lagrangian map for the family of blowup solutions in Theorem \ref{LagrangianIntro} maps the plane $\spn\{v_2,v_3\}$ to infinity at $t\to T_{max}$
and maps the axis $\spn\{v_1\}$ to the origin as $t\to T_{max}$.
All the points that are not in this plane or on that axis evolve according to a linear combination of these evolutions.
This means that the Lagrangian map for the family of blowup solutions involves unbounded planar stretching and axial compression globally. As previously mentioned, a regularity criterion on $\lambda_2^+$ proven first by Neustupa and Penel \cites{NeuPen1,NeuPen2} and by different methods by the author in \cite{MillerStrain} shows that blowup for the Navier--Stokes equation requires unbounded planar stretching and axial compression. For mild solutions with finite-time blowup, this will not be a global phenomenon, with an entire axis compressing to the origin and an entire plane stretching to infinity. The global character of this stretching/compression is a result of the linear dependence on the spatial variables of our solution. Nonetheless, the blowup for this toy model does seem to have some of the features required for actual Navier--Stokes blowup, from the alignment of the vorticity with the intermediate eigenvector to the presence of unbounded planar stretching and axial compression.
\end{remark}

In section 3, we will prove Theorem \ref{BlowupFamIntro}, dealing with the special family of blowup solutions. In section 4, we will consider the behaviour of this special family of solution in the Lagrangian framework, proving Theorem \ref{LagrangianIntro}. In section 5, we will show that the the family of blowup solutions with the vorticity and aligned with the middle eigenvector is stable and is in fact an attractor for a wide range of initial conditions, proving Theorem \ref{GeneralCase}.
Finally, in appendix A, we will prove Theorem \ref{MainTheorem}. The proof of this theorem is left to an appendix, because it only covers the case where the vorticity is aligned with one of the eigenvectors of the strain, and therefore is is less general than Theorem \ref{GeneralCase}. We will still prove this result separately, because in this special case we are able to get more precise asymptotics than in the general case, particularly for the boundary case, $k_0^2=\frac{1}{2}(r_0+2)^2, r_0>0$.

\section{Family of special solutions and phase space analysis}

In this section, we will prove Theorem \ref{BlowupFamIntro} and will also perform phase space analysis for the parameters $r,k$ that will be useful in the proof of Theorem \ref{MainTheorem}.
We will begin by proving Theorem \ref{BlowupFamIntro}, which is restated here for the reader's convenience.

\begin{theorem} \label{BlowupFam}
Suppose we have initial data
\begin{equation}
    S^0=
    \left(
    \begin{array}{ccc}
         -(r+1)\lambda_0& 0 & 0  \\
          0& r \lambda_0 &0 \\
          0 &0 &\lambda_0 
    \end{array}
    \right)
\end{equation}
and 
\begin{equation}
    \omega^0= 
    \left(
    \begin{array}{c}
         0  \\
         2k\lambda_0 \\
         0
    \end{array}
    \right),
\end{equation}
where $\lambda_0>0, -\frac{1}{2}\leq r\leq 1,$ and
$k^2=1+r-2r^2$.
If we let
\begin{equation}
    \partial_t\lambda=r\lambda^2,
\end{equation}
then the pair $(S,\omega)$ given by 
\begin{equation}
    S(t)=
    \left(
    \begin{array}{ccc}
         -(r+1)\lambda(t)& 0 & 0  \\
          0& r \lambda(t) &0 \\
          0 &0 &\lambda(t)
    \end{array}
    \right)
\end{equation}
and
\begin{equation}
        \omega(t)= 
    \left(
    \begin{array}{c}
         0  \\
         2k\lambda(t) \\
         0
    \end{array}
    \right)
\end{equation}
is a solution to the strain-vorticity pair equation
with initial data $\left(S^0,\omega^0\right)$.

The general form for this solution is given by
\begin{equation}
    \lambda(t)=\frac{\lambda_0}{1-r\lambda_0 t}.
\end{equation}
When $r>0$ this solution blows up in finite time,
$T_{max}= \frac{1}{r\lambda_0}$.
When $r=0$, this is a stationary solution, and hence $T_{max}=+\infty.$
When $r<0, T_{max}=+\infty$ and this solution decays to $0$ at infinity.
\end{theorem}

\begin{proof}
We must show that when $k=1+r-2r^2$ and $\partial_t\lambda=r\lambda,$ then $S,\omega$ is a solution of the strain-vorticity pair equation.

First we will observe that 
\begin{equation}
    \partial_t S=
    \left(
    \begin{array}{ccc}
         -(r+1)r\lambda^2& 0 & 0  \\
          0& r^2 \lambda^2 &0 \\
          0 &0 &r\lambda^2
    \end{array}
    \right).
\end{equation}
We can also see that
\begin{align}
    |S|^2&=\left(2+2r+2r^2\right)\lambda^2 \\
    |\omega|^2&=4k^2\lambda^2.
\end{align}
Therefore we can compute that
\begin{equation}
    S^2-\frac{1}{3}|S^2|I_3
    =
    \lambda^2 \left(
    \begin{array}{ccc}
         \frac{1}{3}r^2+\frac{4}{3}r+\frac{1}{3}
         & 0 & 0  \\
          0& \frac{1}{3}r^2-\frac{2}{3}r-\frac{2}{3}
          &0 \\
          0 &0 &-\frac{2}{3}r^2-\frac{2}{3}r+\frac{1}{3}
    \end{array}
    \right).
\end{equation}
Likewise, using the fact that $k^2=1+r-2r^2,$ we may compute that
\begin{equation}
    \frac{1}{4}\omega\otimes\omega
    -\frac{1}{12}|\omega|^2I_3
    =
    \lambda^2\left(
    \begin{array}{ccc}
         \frac{2}{3}r^2-\frac{1}{3}r-\frac{1}{3}
         & 0 & 0  \\
          0& -\frac{4}{3}r^2+\frac{2}{3}r+\frac{2}{3}
          &0 \\
          0 &0 &\frac{2}{3}r^2-\frac{1}{3}r-\frac{1}{3}
    \end{array}
    \right).
\end{equation}
Putting these together, we can compute that
\begin{equation}
    S^2-\frac{1}{3}|S^2|I_3
    +\frac{1}{4}\omega\otimes\omega
    -\frac{1}{12}|\omega|^2I_3
    =
    \lambda^2\left(
    \begin{array}{ccc}
         r^2+r& 0 & 0  \\
          0& -r^2  &0 \\
          0 &0 &-r
    \end{array}
    \right).
\end{equation}
Therefore we can conclude that
\begin{equation}
    \partial_t S+S^2-\frac{1}{3}|S^2|I_3
    +\frac{1}{4}\omega\otimes\omega
    -\frac{1}{12}|\omega|^2I_3
    =0.
\end{equation}

Doing the analogous calculation for the vorticity, we can see that
\begin{equation}
    \partial_t\omega
    =
    \left(
    \begin{array}{c}
         0  \\
         2kr\lambda^2 \\
         0
    \end{array}
    \right),
\end{equation}
and that likewise
\begin{equation}
    S\omega 
    =
    \left(
    \begin{array}{c}
         0  \\
         2kr\lambda^2 \\
         0
    \end{array}
    \right).
\end{equation}
We may therefore conclude that 
\begin{equation}
    \partial_t\omega-S\omega=0,
\end{equation}
and this completes the proof.
\end{proof}

Now we will prove several results involving the parameters $r,k,\lambda$ that will be useful when we prove Theorem \ref{MainTheorem} in the next section.

\begin{proposition}  \label{ParameterODE}
Suppose $\lambda, k,r\in\mathbb{R}\in 
C^1\left([0,T_{max})\right), \lambda> 0$
satisfy the differential equations
\begin{align}
    \partial_t\lambda &= \label{A}
    \frac{1}{3}\left(-1+2r+2r^2+k^2\right) \lambda^2 \\
    \partial_t r &= \label{B}
    \frac{1}{3}\lambda\left(2+3r-3r^2-2r^3-(r+2)k^2\right) \\
    \partial_t k &=\frac{1}{3}\lambda (1+r-2r^2-k^2)k. \label{C}
\end{align}
Then the pair $S,\omega\in C^1\left([0,T_{max})\right)$
is a solution of the strain-vorticity pair equation where
the strain is given by
\begin{equation}
    S(t)=
    \left(
    \begin{array}{ccc}
         -(r(t)+1)\lambda(t)& 0 & 0  \\
          0& r(t) \lambda(t) &0 \\
          0 &0 &\lambda(t) 
    \end{array}
    \right),
\end{equation}
and a vorticity given by
\begin{equation}
    \omega(t)= 
    \left(
    \begin{array}{c}
         0  \\
         2k(t)\lambda(t) \\
         0
    \end{array}
    \right),
\end{equation}
We will say that $\lambda,r,k$ satisfy the strain-vorticity parameter ODE if the differential equations \eqref{A},\eqref{B},\eqref{C} hold with $g$ defined as above.
\end{proposition}

If we let $g=1+r-2r^2-k^2$, then these equations can be expressed nicely in terms of $g$ by
\begin{align}
    \partial_t\lambda&= 
    \left(r-\frac{1}{3}g\right)\lambda^2 \\
    \partial_t r&= 
    \frac{1}{3}\lambda(r+2)g \\
    \partial_t k&= 
    \frac{1}{3}\lambda gk.
\end{align}

\begin{remark}
Note that $g$ measures the defect in the condition $1+r-2r^2-k^2=0$ from Theorem \ref{BlowupFam}.
If $g=0$, then $\partial_t \lambda=r\lambda$, 
and $\partial_t r,\partial_t k=0$. 
It is easy to see that the condition $k^2=1+r-2r^2$ 
can be expressed equivalently as $g=0,$
so Theorem \ref{BlowupFam} can be seen as the special case of Proposition \ref{ParameterODE} where $g=0$, and therefore $r$ and $k$ are constant. 
\end{remark}

\begin{proof}
\begin{align}
    \partial_t\lambda &=
    \frac{1}{3}\left(-1+2r+2r^2+k^2\right) \lambda^2 \\
    \partial_t r &=
    \frac{1}{3}\lambda\left(2+3r-3r^2-2r^3-(r+2)k^2\right) \\
    \partial_t k &=\frac{1}{3}\lambda (1+r-2r^2-k^2)k.
\end{align}
We will begin by computing
\begin{equation}
    S^2= \lambda^2
    \left(\begin{array}{ccc}
         r^2+2r+1& 0 & 0  \\
          0& r^2 &0 \\
          0 &0 & 1
    \end{array}
    \right),
\end{equation}
and likewise that
\begin{equation}
    |S|^2=(2+2r+2r^2)\lambda^2.
\end{equation}
Therefore we can see that
\begin{equation}
    S^2-\frac{1}{3}|S|^2 I_3
    =\lambda^2 \left(\begin{array}{ccc}
         \frac{1}{3}+\frac{4}{3}r+\frac{1}{3}r^2
         & 0 & 0  \\
          0& -\frac{2}{3}-\frac{2}{3}r+\frac{1}{3}r^2
          &0 \\
          0 &0 & \frac{1}{3}-\frac{2}{3}r-\frac{2}{3}r^2
    \end{array}
    \right).
\end{equation}
Making the analogous calculation for the vorticity, we find that
\begin{equation}
    \frac{1}{4}\omega\otimes\omega
    -\frac{1}{12}|\omega|^2I_3
    =
    \lambda^2 \left(\begin{array}{ccc}
         -\frac{1}{3}k^2 & 0 & 0  \\
          0& \frac{2}{3}k^2&0 \\
          0 &0 & -\frac{1}{3}k^2.
    \end{array}
    \right).
\end{equation}
Putting these equations together we find that
\begin{equation}
    S^2-\frac{1}{3}|S|^2 I_3
    +\frac{1}{4}\omega\otimes\omega
    -\frac{1}{12}|\omega|^2I_3
     =\frac{1}{3}\lambda^2 
     \left(\begin{array}{ccc}
         1+4r+r^2-k^2 & 0 & 0  \\
          0& -2-2r+r^2+2k^2 &0 \\
          0 &0 & 1-2r-2r^2-k^2
    \end{array}
    \right).
\end{equation}

Now we will apply the strain vorticity parameter ODE directly to compute the time derivative of the strain, finding that 
\begin{align}
    \partial_t S
    &=
        \left(
    \begin{array}{ccc}
         -r'\lambda-(r+1)\lambda'& 0 & 0  \\
          0& r'\lambda+r\lambda' &0 \\
          0 &0 &\lambda' 
    \end{array}
    \right) \\
    &=
    \frac{\lambda^2}{3}\left(
    \begin{array}{ccc}
    -1-4r-r^2+k^2& 0 & 0  \\
    0& 2+2r-r^2-2k^2 &0 \\
    0 &0 & -1+2r+2r^2+k^2 
    \end{array}
    \right) 
\end{align}
Finally we can conclude that
\begin{equation}
    \partial_t S
    +S^2-\frac{1}{3}|S|^2 I_3
    +\frac{1}{4}\omega\otimes\omega
    -\frac{1}{12}|\omega|^2I_3
    =0.
\end{equation}

Now we will turn to the vorticity equation.
It is clear that
\begin{equation}
    S\omega= 
       \left(
    \begin{array}{c}
         0  \\
         2rk\lambda^2 \\
         0
    \end{array}
    \right),
\end{equation}
Now taking the time derivate and using the parameter ODE, we find that
\begin{align}
    \partial_t\omega
    &=
         \left(
    \begin{array}{c}
         0  \\
         2k'\lambda+2k\lambda' \\
         0
    \end{array}
    \right) \\
    &=
    \left(
    \begin{array}{c}
         0  \\
         2rk\lambda^2 \\
         0
    \end{array}
    \right).
\end{align}
Therefore we can conclude that
\begin{equation}
    \partial_t\omega-S\omega=0,
\end{equation}
and this completes the proof.
\end{proof}

\begin{proposition} \label{PhaseSpace}
Suppose $\lambda,r,k\in C^1\left([0,T_{max})\right)$ 
is a solution of the strain-vorticity parameter system of ODEs.
If $r_0\neq-2,$ then for all $0<t<T_{max}$,
\begin{equation}
    k(t)=m_0(r(t)+2),
\end{equation}
where
\begin{equation}
    m_0=\frac{k_0}{r_0+2}.
\end{equation}
If $r_0=-2,$ then for all $0<t<T_{max}$,
\begin{equation}
    r(t)=-2.
\end{equation}
\end{proposition}

\begin{proof}
We will begin by observing that
\begin{equation}
    \partial_t(r+2)=\frac{1}{3}\lambda g(r+2).
\end{equation}
For all $0<t<T_{max}, \frac{1}{3}\lambda g
\in L^1\left([0,t]\right)$, so we can observe that 
for all $0<t<T_{max}$,
\begin{equation}
    r(t)+2=(r_0+2)\exp\left(\frac{1}{3}\int_0^t
    \lambda(\tau)g(\tau)\diff\tau\right).
\end{equation}
We will note that this implies that if 
$r_0=-2,$ then for all $0<t<T_{max}, r(t)=-2$, so we are now done with the case where $r_0=-2$.
We can also observe that 
if $r_0>-2,$ then for all $0<t<T_{max}, r(t)>-2,$ and
$r_0<-2,$ then for all $0<t<T_{max}, r(t)<-2$

Now suppose that $r_0\neq -2$, and let
\begin{equation}
    m(t)=\frac{k(t)}{r(t)+2}.
\end{equation}
Note that we have just shown that if $r_0\neq -2$,
then for all $0<t<T_{max}, r(t)\neq -2$,
so it is clear that $m(t)$ is well defined.
Using the strain-vorticity parameter ODE to compute the derivative of $m$, we find that
\begin{align}
    \partial_t m
    &=
    \frac{(r+2) k'-k r'}{(r+2)^2} \\
    &=
    \frac{1}{(r+2)^2}
    \left(\frac{1}{3}(r+2)\lambda gk
    -\frac{1}{3}(r+2)\lambda gk\right) \\
    &=0.
\end{align}
Therefore we can see that for all $0<t<T_{max}$
\begin{equation}
    m(t)=m_0.
\end{equation}
This implies that for all $0<t<T_{max}$,
\begin{equation}
    \frac{k(t)}{r(t)+2}=m_0,
\end{equation}
and this completes the proof.
\end{proof}

\begin{corollary} \label{SimplifiedSystem}
Suppose we have initial data
\begin{equation}
    S^0=
    \left(
    \begin{array}{ccc}
         -(r+1)\lambda_0& 0 & 0  \\
          0& r \lambda_0 &0 \\
          0 &0 &\lambda_0 
    \end{array}
    \right)
\end{equation}
and 
\begin{equation}
    \omega^0= 
    \left(
    \begin{array}{c}
         0  \\
         2k\lambda_0 \\
         0
    \end{array}
    \right),
\end{equation}
where $\lambda_0>0, r_0>0$.
Then the pair $S,\omega\in C^1\left([0,T_{max})\right)$
is a solution of the strain-vorticity pair equation where
the strain is given by
\begin{equation}
    S(t)=
    \left(
    \begin{array}{ccc}
         -(r(t)+1)\lambda(t)& 0 & 0  \\
          0& r(t) \lambda(t) &0 \\
          0 &0 &\lambda(t) 
    \end{array}
    \right),
\end{equation}
and a vorticity given by
\begin{equation}
    \omega(t)= 
    \left(
    \begin{array}{c}
         0  \\
         2m_0(r(t)+2)\lambda(t) \\
         0
    \end{array}
    \right),
\end{equation}
where we again take $m_0=\frac{k_0}{r_0+2}$,
and the parameters $\lambda,r$ satisfy the differential equations
\begin{align}
        \partial_t\lambda &=
    \frac{1}{3}\left(-1+2r+2r^2+m_0^2(r+2)^2\right)
    \lambda^2, \\
    \partial_t r &=
    \frac{1}{3}\lambda(r+2)
    \left(1+r-2r^2-m_0^2(r+2)^2\right)
\end{align}
\end{corollary}

\begin{proof}
This follows immediately from Proposition \ref{ParameterODE} and Proposition \ref{PhaseSpace}.
\end{proof}

\section{Particle trajectories in the Lagrangian formulation}
\label{Lagrangian}

In this section we will consider the trajectories of particles in the Lagrangian formulation for the family of blowup solutions in Theorem \ref{BlowupFamIntro}.
We will begin by considering the eigenvalues and eigenvectors of matrices $M$ of the type in Theorem \ref{BlowupFamIntro}.

\begin{proposition} \label{Eigenvectors}
Let $k=1+r-2r^2, 0<r\leq 1$. Then the matrix 
\begin{equation}
    M=
    \left( \begin{array}{ccc}
         -(1+r)& 0& k  \\
         0& r& 0 \\
         -k & 0 &1
    \end{array}
    \right)
\end{equation}
has eigenvalues $\rho=-2r$ and $\rho=r$, where the latter has multiplicity two.
The eigenvector for the eigenvalue $\rho=-2r$ is
\begin{equation}
    v_1=
    \frac{1}{\sqrt{k^2+(1+2r)^2}}
    \left( \begin{array}{c}
         1+2r \\
         0 \\
         k
    \end{array}
    \right),
\end{equation}
and the eigenvectors for the eigenvalue $\rho=r$ are
\begin{equation}
    v_2=
        \left( \begin{array}{c}
         0 \\
         1 \\
         0
    \end{array}
    \right),
\end{equation}
and
\begin{equation}
    v_3=
    \frac{1}{\sqrt{k^2+(1+2r)^2}}
    \left( \begin{array}{c}
         k \\
         0 \\
         1+2r
    \end{array}
    \right).
\end{equation}
Furthermore, $\{v_1,v_2,v_3\}$ is a basis for $\mathbb{R}^3$.
\end{proposition}

\begin{proof}
It is elementary linear algebra to compute, using the condition $k^2=1+r-2r^2$, that
\begin{equation}
    \det\left(\rho I_3-M\right)
    =
    (\rho+2r)(\rho-r)^2,
\end{equation}
and that
\begin{align}
    Mv_1 &= -2rv_1 \\
    Mv_2 &=r v_2 \\
    Mv_3 &= rv_3.
\end{align}
Finally, we observe that the set 
$\{v_1,v_2,v_3\}$ is linearly independent, 
and therefore forms of a basis of $\mathbb{R}^3$,
and this completes the proof.
\end{proof}

We will now prove Theorem \ref{LagrangianIntro}, with the result broken into two pieces.

\begin{theorem} \label{LangragianThmBody}
Suppose $y_0=c_1 v_1+c_2v_2+c_3v_3$, and
\begin{equation}
    \partial_t y(t)=u(y(t),t),
\end{equation}
where 
\begin{equation}
    u(x,t)=M(t)x,
\end{equation}
and for all $0\leq t<T_{max}=\frac{1}{r\lambda_0}$,
\begin{equation}
    M(t)=\frac{\lambda_0}{1-r\lambda_0 t}
    \left(\begin{array}{ccc}
         -(1+r)& 0 & k  \\
         0& r & 0 \\
         -k &0 & 1
    \end{array}
    \right).
\end{equation}
Then for all $0\leq t<T_{max}$,
\begin{equation}
    y(t)=c_1 (1-r\lambda_0 t)^2v_1
    +c_2\frac{1}{1-r\lambda_0t}v_2
    +c_3 \frac{1}{1-r\lambda_0t} v_3.
\end{equation}
\end{theorem}

\begin{proof}
We will begin by letting
\begin{equation} \label{LagrangianMap}
    y(t)=
    c_1 (1-r\lambda_0 t)^2v_1
    +c_2\frac{1}{1-r\lambda_0t}v_2
    +c_3 \frac{1}{1-r\lambda_0t} v_3,
\end{equation}
for all $0\leq t<T_{max}$.
It is clear that $y(0)=y_0$, and that
$u(y(t),t)=M(t)y(t)$, so it suffices to show that
for all $0\leq t<T_{max}$,
\begin{equation}
    \partial_t y(t)=M(t)y(t).
\end{equation}
Applying Proposition \ref{Eigenvectors},
we can see that
\begin{align}
    M(t)v_1 &= \frac{-2r\lambda_0}{1-r\lambda_0 t}v_1 \\
    M(t)v_2 &= \frac{r\lambda_0}{1-r\lambda_0 t}v_2 \\
    M(t)v_3 &= \frac{r\lambda_0}{1-r\lambda_0 t}v_3.
\end{align}
Therefore we can compute that
\begin{equation}
    M(t)y(t)=-2c_1r\lambda_0(1-r\lambda_0 t)v_1
    +c_2\frac{r\lambda_0}{(1-r\lambda_0 t)^2}v_2
    ++c_3\frac{r\lambda_0}{(1-r\lambda_0 t)^2}v_3
\end{equation}
Likewise, differentiating \eqref{LagrangianMap}, we find that
\begin{equation}
    \partial_t y(t)=
    -2c_1r\lambda_0(1-r\lambda_0 t)v_1
    +c_2\frac{r\lambda_0}{(1-r\lambda_0 t)^2}v_2
    ++c_3\frac{r\lambda_0}{(1-r\lambda_0 t)^2}v_3.
\end{equation}
We have now verified that for all $0\leq t<T_{max}$,
\begin{equation}
    \partial_t y(t)=u(y(t),t),
\end{equation}
and this completes the proof.
\end{proof}

\begin{corollary} \label{LagrangeCor}
Let $Y(y_0,t)$ be the Lagrangian map associated with the velocity field
\begin{equation}
    u(x,t)=M(t)x,
\end{equation}
where for all $0\leq t<T_{max}=\frac{1}{r\lambda_0}$,
\begin{equation}
    M(t)=\frac{\lambda_0}{1-r\lambda_0 t}
    \left(\begin{array}{ccc}
         -(1+r)& 0 & k  \\
         0& r & 0 \\
         -k &0 & 1
    \end{array}
    \right).
\end{equation}
That is let
\begin{equation}
    \partial_t Y(y_0,t)=u(Y(y_0,t),t)
\end{equation}
and
\begin{equation}
    Y(y_0,0)=y_0.
\end{equation}
Then for all $0\leq t<T_{max}$,
\begin{equation}
    Y(y_0,t)=QD(t)Q^{-1}y_0,
\end{equation}
where 
\begin{equation}
    D(t)=\left(\begin{array}{ccc}
         (1-r\lambda_0t)^2 &0 &0  \\
         0 & \frac{1}{1-r\lambda_0t} &0 \\
         0 &0 & \frac{1}{1-r\lambda_0t}
    \end{array}
    \right),
\end{equation}
and
\begin{equation}
    Q=\left(\begin{array}{ccc}
         \frac{1+2r}{\sqrt{k^2+(1+2r)^2}} &0 &\frac{k}{\sqrt{k^2+(1+2r)^2}}  \\
         0 & 1 &0 \\
         \frac{k}{\sqrt{k^2+(1+2r)^2}} &0 & \frac{1+2r}{\sqrt{k^2+(1+2r)^2}}
    \end{array}
    \right),
\end{equation}
\end{corollary}

\begin{proof}
We will begin by setting
\begin{equation}
    Y(y_0,t)=QD(t)Q^{-1}y_0,
\end{equation}
for all $y_0\in\mathbb{R}^3, 0\leq t<T_{max}$.
It is easy to see that for all 
$y_0\in\mathbb{R}^3, Y(y_0,0)=y_0,$
so it only remains to show that for all $y_0\in\mathbb{R}^3,$
and for all $0\leq t<T_{max}$, 
\begin{equation}
    \partial_t Y(y_0,t)=u(Y(y_0,t),t)
\end{equation}
We will begin by fixing $y_0\in\mathbb{R}^3$.
Applying Proposition \ref{Eigenvectors}, 
using the fact that $\{v_1,v_2,v_3\}$ 
is a basis for $\mathbb{R}^3$,
let $y_0=c_1v_1+c_2v_2+c_3v_3$.

Next we observe that $Q$ is the change of basis matrix for the change from the eigenbasis $\{v_1,v_2,v_3\}$
to the standard basis $\{e_1,e_2,e_3\}$.
Likewise we can see that $Q^{-1}$ is the change of basis matrix for the change from the standard basis $\{e_1,e_2,e_3\}$
to the eigenbasis $\{v_1,v_2,v_3\}$.
Therefore it is clear that
\begin{equation}
    Q^{-1}y_0=\left(\begin{array}{c}
         c_1  \\
          c_2 \\
          c_3
    \end{array}
    \right)
\end{equation}
gives the coordinates of $y_0$ in the eigenbasis,
and so
\begin{equation}
    y_0=c_1v_1+c_2v_2+c_3v_3
\end{equation}
We can also compute that
\begin{equation}
    QD(t)Q^{-1}y_0=
    c_1 (1-r\lambda_0 t)^2v_1
    +c_2\frac{1}{1-r\lambda_0t}v_2
    +c_3 \frac{1}{1-r\lambda_0t} v_3,
\end{equation}
Applying Theorem \ref{LangragianThmBody}, we can conclude that
\begin{equation}
    \partial_t Y(y_0,t)=u(Y(y_0,t),t),
\end{equation}
and this completes the proof.
\end{proof}

We will now consider the evolution of several circles under the map $Y(\cdot,t)$. These examples will provide some insight into the geometric character of the flow.
The circle in the plane given by $\spn\{v_2,v_3\}$ is stretched out towards infinity by the flow. The stretching out of the circle is the only deformation, as the shape of the circle is preserved.

\begin{proposition} \label{PropExpand}
Let $C_{R}$ be the circle in the plane $\spn\{v_2,v_3\}$
with radius $R$,
\begin{equation}
    C_R=\{R\cos(\theta)v_2+R\sin(\theta)v_3:
    -\pi<\theta\leq \pi\},
\end{equation}
and let $Y(\cdot,t)$ be the Lagrangian map from Corollary \ref{LagrangeCor}.
Then for all $0\leq t<T_{max}$, 
\begin{equation}
    Y(\cdot,t)(C_{R_0})=C_{\frac{R_0}{1-r\lambda_0 t}}
    =\left\{\frac{R_0}{1-r\lambda_0 t}\cos(\theta)v_2
    +\frac{R_0}{1-r\lambda_0 t}\sin(\theta)v_3:
    -\pi<\theta\leq \pi\right\}.
\end{equation}
\end{proposition}

\begin{proof}
Applying Theorem \ref{LangragianThmBody},
we can see that for all $0\leq t<T_{max}$,
\begin{equation}
    Y(R\cos(\theta)v_2+r\sin(\theta)v_3,t)=
    \frac{R_0}{1-r\lambda_0 t}\cos(\theta)v_2
    +\frac{R_0}{1-r\lambda_0 t}\sin(\theta)v_3.
\end{equation}
This completes the proof.
\end{proof}

Now we will consider the evolution a circle in the $yz$-plane under the flow $Y(\cdot,t)$. When $0<r<1$, this circle is not aligned with the eigenbasis, and so its behaviour will be more complicated. It will rotate and deform, in addition to spreading out to infinity.

\begin{proposition} \label{PropRotate}
Let $C_R$ be the unit circle in the $yz$-plane $\spn\{e_1,e_2\}$,
with radius $R$,
\begin{equation}
    C_R=\{R\cos(\theta)e_2+R\sin(\theta)e_3:
    -\pi<\theta\leq \pi\},
\end{equation}
and let $Y(\cdot,t)$ be the Lagrangian map from Corollary \ref{LagrangeCor}.
Furthermore, let
\begin{align}
    c_1 &= -\frac{k\sqrt{k^2+(1+2r)^2}}{3r(1+2r)} \\
    c_3 &= \frac{\sqrt{k^2+(1+2r)^2}}{3r}
\end{align}
Then for all $0\leq t<T_{max}$, 
\begin{equation}
    Y(\cdot,t)(C_{R_0})
    =\left\{c_1R_0(1-r\lambda_0 t)^2 v_1
    +\frac{R_0}{1-r\lambda_0 t}\cos(\theta)v_2
    +c_3\frac{R_0}{1-r\lambda_0 t}\sin(\theta)v_3:
    -\pi<\theta\leq \pi\right\}.
\end{equation}
\end{proposition}.

\begin{proof}
We will begin by changing from the standard basis to the eigenbasis. We will compute that 
\begin{align}
    e_2&=v_2 \\
    e_3&=c_1v_1+c_3v_3.
\end{align}
Therefore we can see that
\begin{equation}
    C_{R_0}=\{c_1 R_0\sin(\theta)v_1
    +R_0\cos(\theta)v_2+c_3R_0\sin(\theta)v_3:
    -\pi<\theta\leq \pi\},
\end{equation}
and applying Theorem \ref{LangragianThmBody},
we find that for all $0\leq t<T_{max}$,
\begin{equation}
    Y(c_1 R_0\sin(\theta)v_1
    +R_0\cos(\theta)v_2+c_3R_0\sin(\theta)v_3,t)
    =
    R_0(1-r\lambda_0 t)^2 v_1
    +\frac{R_0}{1-r\lambda_0 t}\cos(\theta)v_2
    +c_3\frac{R_0}{1-r\lambda_0 t}\sin(\theta)v_3.
\end{equation}
This completes the proof.
\end{proof}

\begin{remark}
In Remark \ref{RotationRemarkIntro}, we discussed the apparent lack of rotation of the Lagrangian map---even in the case with vorticity when $0<r<1$---and explained that the rotation is embedded in the fact that the eigenbasis is no longer orthogonal.
The example in Proposition \ref{PropRotate} we can very clearly see that rotation is occurring. The circle begins in the $yz$-plane, but rotates about the $y$-axis as it evolves, until it approaches the plane spanned by ${v_2,v_3}$ as $t\to T_{max}$.
This rotation is generated by the vorticity, which is in the in the $y$-direction, so rotation generated by the vorticity is in fact present in the Lagrangian formulation.
\end{remark}

\begin{remark}
We should also note that we can see from the example in Proposition \ref{PropRotate} the fact that $v_2=e_2$ is the intermediate eigenvalue of the strain $S$ for $0<r<1$. In Proposition \ref{PropExpand}, the circle expands in the plane given by $\spn\{v_2,v_3\}$ with no other deformation. In this sense the fact that the deformation matrix is stretching stronger in the $z$ direction than the $y$ direction is obscured. However, when we start with a circle in the $yz$-plane, it still maps to infinity in the plane given by the $\spn\{v_2,v_3\}$, but this time it is asymptotically an ellipse, not a circle. As $t\to T_{max}$, $Y(\cdot,t)(C_R)$ expands out to infinity and its shape asymptotically is an ellipse, with $v_3$ as the major axis and $v_2=e_2$ as the minor axis,
because $c_3>1$.
The fact that the $y$-axis is the minor axis corresponds to the fact that the stretching due to the strain is weaker in this direction than in $z$-axis.
This connection between variable rates of stretching and rotation shows the central importance of the alignment of the vorticity with the eigenvector corresponding to the intermediate eigenvalue of the strain in the Lagrangian dynamics.
\end{remark}

\begin{remark}
We will now show that for the blowup solutions in Theorem \ref{BlowupFamIntro}, the Seregin and S\v{v}er\'ak condition is satisfied, with $p$ becoming unbounded below and 
$p+\frac{1}{2}|u|^2$ becoming unbounded above
as $t\to T_{max}$.
Suppose $u(x,t)=M(t)x$, where
\begin{equation}
    M(t)=\frac{\lambda_0}{1-r\lambda_0 t}
    \left(\begin{array}{ccc}
         -(1+r)& 0 & k  \\
         0& r & 0 \\
         -k &0 & 1
    \end{array}
    \right),
\end{equation}
with $0< r\leq 1$ and
$k^2=1+r-2r^2$.
Note that $u(x,t)$ is one of the solutions from the family of finite-time blowup solutions in Theorem \ref{BlowupFamIntro}, and $u$ blows up in at time $T_{max}=\frac{1}{r\lambda_0}.$
From Proposition \ref{Eigenvectors}, we know that the eigenvalues of $M(t)$ are $\frac{-2r\lambda_0}{1-r\lambda_0t}, 
\frac{r\lambda_0}{1-r\lambda_0t},
\frac{r\lambda_0}{1-r\lambda_0t}$,
and so 
\begin{equation}
    \frac{1}{6}\tr\left(M(t)^2\right)=
    \frac{r^2\lambda_0^2}{(1-r\lambda_0t)}.
\end{equation}
From this we may conclude that
\begin{equation}
    p(x,t)=-\frac{r^2\lambda_0^2}{(1-r\lambda_0t)}|x|^2,
\end{equation}
and so $p$ becomes unbounded from below as $t\to T_{max}$.
Likewise, again applying Proposition \ref{Eigenvectors},
we find that 
\begin{equation}
    M(t)v_1=-\frac{2r\lambda_0}{(1-r\lambda_0 t)}v_1.
\end{equation}
Considering the velocity at points along the $v_1$ axis,
we compute that
\begin{align}
    u(|x|v_1,t)
    &=
    M(t)|x|v_1 \\
    &=
    -|x|\frac{2r\lambda_0}{(1-r\lambda_0 t)}v_1
\end{align}
In turn we can compute that 
\begin{equation}
    \left(p+\frac{1}{2}|u|^2\right)(|x|v_1,t)=
    \frac{r^2\lambda_0^2}{(1-r\lambda_0t)}|x|^2,
\end{equation}
and so $p+\frac{1}{2}|u|^2$ becomes unbounded above 
as $t\to T_{max}$.
\end{remark}

\section{Stability of the blowup: the general case}.

In this section we will prove Theorem \ref{GeneralCase}, proving the finite-time blowup of solutions of the matrix Navier--Stokes equation for all initial data $M^0$ such that the eigenvalues of $M$ satisfy
\begin{equation}
    \mathrm{Re}(\lambda_1)<\mathrm{Re}(\lambda_2) 
    \leq \mathrm{Re}(\lambda_3).
\end{equation}
We will break this down into three cases. First, the case where $M^0$ is diagonalizable over $\mathbb{R}$, in which case the result reduces to one previously proven by the author in \cites{MillerStrain}. Second, the case where $M^0$ is diagonalizable over $\mathbb{C}$, but not $\mathbb{R}$, and in which the complex eigenvalue has a positive real part, in which case the Jordan normal form of $M^0$ must be given by
\begin{equation} \label{ComplexNormalForm}
    J=\left(\begin{array}{ccc}
         -2\lambda& 0 & 0  \\
         0 & \lambda-ia & 0 \\
         0 & 0 & \lambda+ia
    \end{array}\right).
\end{equation}
Finally, we will consider the case where $M^0$ is not diagonalizable over $\mathbb{C}$, and the positive eigenvalue is repeated, meaning that the Jordan normal form of $M^0$ must be given by
\begin{equation} \label{RepeatedNormalForm}
    J=\left(\begin{array}{ccc}
         -2\lambda& 0 & 0  \\
         0 & \lambda & 1 \\
         0 & 0 & \lambda
    \end{array}\right).
\end{equation}

Note that these three cases are the only possible condition because $M^0$ is a real valued matrix, and therefore if it has a complex eigenvalue, then the complex conjugate must also be an eigenvalue. 
This immediately implies that if $\lambda+ia$ is an eigenvalue, then $\lambda-ia$ must be as well, and so if $M^0$ is diagonalizable over $\mathbb{C}$, but not $\mathbb{R}$, then the Jordan normal form must be given by \eqref{ComplexNormalForm}.
Likewise, if $M^0$ is not diagonalizable over $\mathbb{C}$, then its characteristic polynomial must have a repeated root. This repeated root must be real, otherwise its complex conjugate would also be a repeated root, which would lead to a combined algebraic multiplicity of $4$, contradicting 
$M^0\in\mathbb{R}^{3 \times 3}, \tr\left(M^0\right)=0$.
The condition $\tr\left(M^0\right)=0$ then implies that all eigenvalues of $M^0$ are real.
This means that if $M^0$ is not diagonalizable over $\mathbb{C}$ and there is a positive repeated root, then the Jordan normal form must be given by \eqref{RepeatedNormalForm}.

In order to prove Theorem \ref{GeneralCase} in each of these cases, we will first derive an evolution equation for the Jordan normal form.

\begin{proposition} \label{JordanDiffEq}
Suppose the initial data has the Jordan normal form
\begin{equation}
    M^0=QJ^0 Q^{-1}.
\end{equation}
If $J\in C^1\left([0,T_{max});\mathbb{R}^{3\times 3}\right)$
satisfies the equation
\begin{equation}
    \partial_t J+ J^2-\frac{1}{3}\tr\left(J^2\right)I_3=0,
\end{equation}
then 
\begin{equation}
    M(t)=Q J(t) Q^{-1}
\end{equation}
satisfies the matrix Navier--Stokes equation,
where $Q$ is a constant matrix.
\end{proposition}

\begin{proof}
Begin by observing that
\begin{equation}
    \partial_t M(t)= Q\partial_t J(t)Q^{-1},
\end{equation}
We can also see that 
\begin{equation}
    M(t)^2=QJ(t)^2Q^{-1}
\end{equation}
and
\begin{equation}
    \tr\left(M(t)^2\right)=
    \tr\left(J(t)^2\right).
\end{equation}
Therefore can compute that
\begin{align}
    \partial_t M(t)+M(t)^2-\frac{1}{3}\tr\left(M(t)\right)^2
    &=
    Q\partial_t J(t)Q^{-1}+Q J(t)^2Q^{-1}
    -\frac{1}{3}\tr\left(J(t)^2\right)I_3 \\
    &=
    Q \left(
    \partial_t J+ J^2-\frac{1}{3}\tr\left(J^2\right)I_3
    \right) Q^{-1} \\
    &=0,
\end{align}
and this completes the proof.
\end{proof}

\begin{remark}
We will note that Proposition \ref{JordanDiffEq} shows that the generalized eigenbasis for $M^0$ does not evolve in time, it is only the matrix $J(t)$---which is $M(t)$ represented in this eigenbasis---which evolves in time. 
\end{remark}

\begin{theorem}
Suppose $M^0$ is diagonalizable over $\mathbb{R}$ with $\lambda_1<\lambda_2 \leq \lambda_3$. Without loss of generality we will take $M^0=QDQ^{-1}$,
\begin{equation}
    D=\left(
    \begin{array}{ccc}
         -(1+r_0)\lambda_0 &0 &0  \\
         0& r_0\lambda_0 & 0 \\
         0 &0 &\lambda_0
    \end{array}\right).
\end{equation}
with $\lambda_0>0, -\frac{1}{2}<r_0\leq 1$.
Let $M\in C^1\left([0,T_{max});\mathbb{R}^{3\times 3}\right)$
be the solution of the matrix Navier--Stokes equation with initial data $M^0$.
Then $M(t)=QD(t)Q^{-1}$,
where 
\begin{equation}
    D=\left(
    \begin{array}{ccc}
         -1-r(t) &0 &0  \\
         0& r(t) & 0 \\
         0 &0 &1
    \end{array}\right) \lambda(t).
\end{equation}
Furthermore, $M$ blows up at finite-time $T_{max}<+\infty$
with
\begin{equation}
    \lim_{t\to T_{max}} r(t)=1
\end{equation}
and
\begin{equation}
    \lambda(t) \sim \frac{1}{T_{max}-t}
\end{equation}
as $t\to T_{max}$.
\end{theorem}

\begin{proof}
Applying Proposition \ref{JordanDiffEq}, this reduces to proving blowup and these asymptotic estimates for the equation 
\begin{equation}
    \partial_t D(t)+D(t)^2
    -\frac{1}{3} \tr\left(D(t)^2\right)I_3=0.
\end{equation}
The author proved precisely these results in \cites{MillerStrain}, although in that paper using the parameterization
\begin{equation}
    \tilde{r}=-\frac{\lambda_1}{\lambda_3},
\end{equation}
rather than
\begin{equation}
    r=\frac{\lambda_2}{\lambda_3}.
\end{equation}
This difference in parameterization means that the asymptotics are stated as 
\begin{equation}
    \lim_{t\to T_{max}}\tilde{r}(t)=2.
\end{equation}
This is equivalent because
\begin{equation}
    \lim_{t\to T_{max}}
    \frac{-\lambda_1(t)}{\lambda_3(t)}=2
\end{equation}
if and only if
\begin{equation}
    \lim_{t\to T_{max}}
    \frac{\lambda_2(t)}{\lambda_3(t)}=1.
\end{equation}
\end{proof}

\begin{theorem}
Suppose $M^0$ is diagonalizable over $\mathbb{C}$ with eigenvalues $-2\lambda_0, \lambda_0-ia_0, \lambda_0+ia_0$ 
where $\lambda_0>0, a_0>0$.
Then $M^0$ can be written as $M=QDQ^{-1}$ where
\begin{equation}
    D=\left(\begin{array}{ccc}
         -2\lambda_0 & 0 & 0  \\
         0 & \lambda_0-ia_0 &0 \\
         0 &0 & \lambda_0+ia_0
    \end{array} \right).
\end{equation}
Let $M\in C^1\left([0,T_{max});\mathbb{R}^{3\times 3}\right)$
be the solution of the matrix Navier--Stokes equation with initial data $M^0$.
Then $M(t)=Q D(t) Q^{-1}$,
where
\begin{equation}
    D(t)=\left(\begin{array}{ccc}
         -2\lambda(t) & 0 & 0  \\
         0 & \lambda(t)-ia(t) &0 \\
         0 &0 & \lambda(t)+ia(t)
    \end{array} \right).
\end{equation}
Furthermore there is finite-time blowup at $T_{max}<+\infty$
with the asymptotics
\begin{equation}
    \lim_{t\to T_{max}} a(t)=0
\end{equation}
and
\begin{equation}
    \lambda(t) \sim \frac{1}{T_{max}-t},
\end{equation}
as $t \to T_{max}$.
\end{theorem}

\begin{proof}
We will begin by computing that
\begin{equation}
    D^2=\left(\begin{array}{ccc}
         4\lambda^2 & 0 & 0  \\
         0 & \lambda^2-a^2-2i a\lambda  &0 \\
         0 &0 & \lambda^2+\frac{1}{3}a^2+2i a\lambda 
    \end{array} \right),
\end{equation}
and therefore
\begin{equation}
    \frac{1}{3}\tr\left(D^2\right)=2\lambda^2-\frac{2}{3}a^2.
\end{equation}
This means we can compute that
\begin{equation}
    -D^2+\frac{1}{3}\tr\left(D^2\right)I_3
    =\left(\begin{array}{ccc}
         -2\lambda^2 -\frac{2}{3}a^2 & 0 & 0  \\
         0 & \lambda^2+\frac{1}{3}a^2-2ia\lambda  &0 \\
         0 &0 & \lambda^2-a^2+2i \lambda a
    \end{array} \right).
\end{equation}
Applying Proposition \ref{JordanDiffEq}, we know that
\begin{equation}
    \partial_tD+D^2-\frac{1}{3}\tr\left(D^2\right)I_3=0,
\end{equation}
and so we can conclude that $\lambda, a$ satisfy the differential equations
\begin{align}
    \partial_t\lambda &=\lambda^2+\frac{1}{3}a^2 \\
    \partial_ta &=-2\lambda a, \label{EQ3}
\end{align}
Clearly we can see that $\lambda$ is increasing and $a$ is decreasing, and also that for all $0<t<T_{max}$
\begin{equation}
    \partial_t\lambda>\lambda^2.
\end{equation}
Integrating this differential inequality we find
for all $0<t<T_{max}$,
\begin{equation}
    \lambda(t)>\frac{\lambda_0}{1-\lambda_0t},
\end{equation}
and so we have that $T_{max}\leq \frac{1}{\lambda_0}$.

Next we observe that
\begin{equation}
    \partial_t \lambda(t)
    =\left(1+\frac{1}{3}\frac{a^2}{\lambda^2}\right)
    \lambda^2,
\end{equation}
and consequently, integrating this differential inequality from $t$ to $T_{max}$, we can see that for all $0<t<T_{max}$
\begin{equation} \label{EQ4}
    \frac{1}{\lambda(t)}
    =\int_t^{T_{max}}\left(1+\frac{1}{3}
    \frac{a(\tau)^2}{\lambda(\tau)^2}\right) 
    \diff\tau.
\end{equation}
Clearly we have
\begin{equation}
    \lim_{t\to T_{max}} \frac{a(t)}{\lambda(t)}=0,
\end{equation}
and so we can see that
\begin{equation}
    \lim_{t\to T_{max}} (T_{max}-t)\lambda(t)=1,
\end{equation}
and consequently
\begin{equation}
    \lambda(t) \sim \frac{1}{T_{max}-t},
\end{equation}
as $t \to T_{max}$.

Now we will show that $\lim_{t\to T_{max}}a(t)=0$.
Plugging into the differential equation \eqref{EQ3},
we find that tor all $0<t<T_{max}$,
\begin{equation}
    a(t)=a_0\exp\left(-2\int_0^t
    \lambda(\tau) \diff\tau \right).
\end{equation}
Plugging into \eqref{EQ4}, and using the fact that for all $0<t<T_{max}$,
\begin{equation}
    1+\frac{a(t)^2}{3\lambda(t)^2}
    < 
    1+\frac{a_0^2}{3\lambda_0^2},
\end{equation}
we find that for all $0<t<T_{max}$,
\begin{equation}
    \lambda(t)>\frac{1}
    {\left(1+\frac{a_0^2}{3\lambda_0^2}\right)
    (T_{max}-t)}.
\end{equation}
Integrating this upper bound, we find that for all $0<t<T_{max}$,
\begin{equation}
    -2\int_0^t \lambda(\tau) \diff\tau 
    <
    2\frac{2}{\left(1+\frac{a_0^2}{3\lambda_0^2}\right)}
    (T_{max}-t),
\end{equation}
which in turn implies that for all $0<t<T_{max}$,
\begin{equation}
    a(t)<a_0 (T_{max}-t)^\frac{2}
    {\left(1+\frac{a_0^2}{3\lambda_0^2}\right)}.
\end{equation}
This clearly implies that
\begin{equation}
    \lim_{t \to T_{max}} a(t)=0,
\end{equation}
and so this completes the proof.
\end{proof}

\begin{theorem}
Suppose $M^0$ is not diagonalizable over $\mathbb{C}$ and has a positive repeated root $\lambda_0>0$. Then $M^0$ can be written in Jordan normal form as $M=QJQ^{-1}$, where
\begin{equation}
    J=\left(
    \begin{array}{ccc}
         -2\lambda_0 & 0 &0  \\
         0 & \lambda_0 & 1 \\
         0 & 0 & \lambda_0
    \end{array} \right).
\end{equation}
Let $M\in C^1\left([0,T_{max});\mathbb{R}^{3\times 3}\right)$
be the solution of the matrix Navier--Stokes equation with initial data $M^0$.
Then $M(t)=QJ(t)Q^{-1}$,
with
\begin{equation}
    J(t)=\left(\begin{array}{ccc}
         -2\lambda(t)& 0 & 0 \\
         0 & \lambda(t) & r(t) \\ 
         0 & 0 & \lambda(t)
    \end{array}\right),
\end{equation}
where
\begin{align}
    \lambda(t)&=\frac{\lambda_0}{1-\lambda_0 t} \\
    r(t)&= (1-\lambda_0 t)^2.
\end{align}
In particular, this implies that there is blowup at $T_{max}=\frac{1}{\lambda_0}$.
\end{theorem}

\begin{proof}
We will begin by observing that
\begin{equation}
    J^2-\frac{1}{3}\tr\left(J^2\right)I_3
    =
    \left(\begin{array}{ccc}
         2\lambda^2& 0 & 0 \\
         0 & -\lambda^2 & 2\lambda r \\ 
         0 & 0 & -\lambda^2
    \end{array}\right),
\end{equation}
and so we can see that
$\lambda,r$ satisfy the differential equations
\begin{align}
    \partial_t\lambda &=\lambda^2 \label{EQ1} \\
    \partial_t r&=-2\lambda r. \label{EQ2}
\end{align}
Integrating the differential equation \eqref{EQ1}, we find that
\begin{equation}
    \lambda(t)=\frac{\lambda_0}{1-\lambda_0t}.
\end{equation}
Integrating the differential equation \eqref{EQ2}, we find that
\begin{align}
    r(t)&=\exp\left(-2\int_0^t
    \frac{\lambda_0}{1-\lambda_0 \tau}\diff\tau\right) \\
    &=
    \exp\left(2\log(1-\lambda_0 t)\right) \\
    &=
    (1-\lambda_0 t)^2,
\end{align}
and this completes the proof.
\end{proof}

\section*{Acknowledgments}
This publication was supported in part by the Fields Institute for Research in the Mathematical Sciences while the author was in residence during the Fall 2020 semester. Its contents are solely the responsibility of the author and do not necessarily represent the official views of the Institute.
This material is based upon work supported by the National Science Foundation under Grant No. DMS-1440140 while the author participated in a program that was hosted by the Mathematical Sciences Research Institute in Berkeley, California, during the Spring 2021 semester.

\bibliography{Bib}
\bibliographystyle{amsplain}

\appendix

\section{Asymptotics of blowup: the vorticity/eigenframe alignment case}

In this section, we will prove Theorem \ref{MainTheorem}, with each of the different conditions on the parameters listed in Theorem \ref{MainTheorem} considered separately in order to make the proofs easier to follow.
The results in section are less general than the results in section 5, however in certain cases the asymptotics are more precise. See in particular the blowup at $T=+\infty$ in the borderline case in Theorem \ref{BoundaryCase}.

\begin{theorem} \label{BlowupFromBelow}
Suppose we have initial data
\begin{equation}
    S^0=
    \left(
    \begin{array}{ccc}
         -(r_0+1)\lambda_0& 0 & 0  \\
          0& r_0 \lambda_0 &0 \\
          0 &0 &\lambda_0 
    \end{array}
    \right)
\end{equation}
and 
\begin{equation}
    \omega^0= 
    \left(
    \begin{array}{c}
         0  \\
         2k_0\lambda_0 \\
         0
    \end{array}
    \right),
\end{equation}
where $\lambda_0>0, -\frac{1}{2}<r_0< 1$ and 
$k_0^2<1+r_0-2r_0^2$,
and suppose the pair 
$S,\omega\in C^1\left([0,T_{max})\right)$
is a solution of the strain-vorticity pair equation where
the strain is given by
\begin{equation}
    S(t)=
    \left(
    \begin{array}{ccc}
         -(r(t)+1)\lambda(t)& 0 & 0  \\
          0& r(t) \lambda(t) &0 \\
          0 &0 &\lambda(t) 
    \end{array}
    \right),
\end{equation}
and a vorticity given by
\begin{equation}
    \omega(t)= 
    \left(
    \begin{array}{c}
         0  \\
         2 k(t) \lambda(t) \\
         0
    \end{array}
    \right).
\end{equation}
Then $T_{max}<+\infty$ and 
\begin{align*}
    \lim_{t\to T_{max}} r(t)&=r_\infty, \\
    \lim_{t\to T_{max}}k(t)&=k_\infty, \\
    \lim_{t\to T_{max}}g(t)&=0, \\
    \lim_{t\to T_{max}}\lambda(t)&=+\infty,
\end{align*}
where
\begin{align}
    r_\infty&=\frac{1-4m_0^2
    +3\sqrt{1-4m_0^2}}{2m_0^2+4},
    \\
    k_\infty&= m_0 \left(
    \frac{9+3\sqrt{1-4m_0^2}}{2m_0^2+4}\right).
\end{align}
and
\begin{equation}
    m_0=\frac{k_0}{r_0+2}
\end{equation}
\end{theorem}

\begin{proof}
We begin by applying Corollary \ref{SimplifiedSystem}, which implies that the dynamics are completely determined the simplified system
\begin{align}
        \partial_t\lambda &=
    \frac{1}{3}\left(-1+2r+2r^2+m_0^2(r+2)^2\right)
    \lambda^2, \\
    \partial_t r &=
    \frac{1}{3}\lambda(r+2)
    \left(1+r-2r^2-m_0^2(r+2)^2\right).
\end{align}
We will define
\begin{equation}
    f(r)=-1+2r+2r^2+m_0^2(r+2)^2.
\end{equation}
We can see that 
\begin{equation}
    f'(r)=2+4r+m_0^2(r+2)
\end{equation},
and therefore that $f$ is an increasing function
on the interval $-\frac{1}{2}\leq r\leq 1$.
This implies that for all $-\frac{1}{2}\leq r\leq 1$,
\begin{equation}
    f(r)>f\left(-\frac{1}{2}\right)\geq -\frac{3}{2}.
\end{equation}

Therefore we can see that for all $0<t<T_{max},$
\begin{equation}
    \partial_t\lambda(t)\geq -\frac{1}{2}\lambda^2.
\end{equation}
Integrating this differential inequality we find that for all $0<t<T_{max}$,
\begin{equation}
    \lambda(t)\geq \frac{\lambda_0}
    {1+\frac{1}{2}\lambda_0 t}.
\end{equation}
Integrating this inequality we find that 
for all $0<t<T_{max}$,
\begin{equation}
    \int_0^t\lambda(\tau)\diff\tau
    \geq 
    2\log\left(\frac{2}{\lambda_0}+t\right).
\end{equation}

Next we will let
\begin{equation}
    g(r)=1+r-2r^2-m_0^2(r+2)^2.
\end{equation}
We will observe that $g$ has zeroes at 
\begin{equation}
    r=\frac{\left(1-4m_0^2\right) 
    \pm 3\sqrt{1-4m_0^2}}
    {2m_0^2+4}.
\end{equation}
Recall that 
\begin{equation}
    r_\infty=
    \frac{\left(1-4m_0^2\right) 
    + 3\sqrt{1-4m_0^2}}
    {2m_0^2+4},
\end{equation}
and let
\begin{equation}
    r_*=
    \frac{\left(1-4m_0^2\right) 
    - 3\sqrt{1-4m_0^2}}
    {2m_0^2+4}.
\end{equation}
By hypothesis, $g(r_0)>0,$ and therefore
\begin{equation}
    r_*<r_0<r_\infty.
\end{equation}
Furthermore, we can see that if $r_*<r(t)<r_\infty,$
then $r'(t)>0$.

We will now show that for all $0<t<T_{max}, r_0<r(t)<r_\infty,$ and $r$ is an increasing function.
Given the two zeroes, $r_*,r_\infty$,
we can rewrite $g$ in the form
\begin{equation}
    g(r)=C(r-r_*)(r-r_\infty),
\end{equation}
for some $C\in\mathbb{R}$.
Plugging in for $g(1)$ we find
\begin{equation}
    g(0)=Cr_*r_\infty=1-4m_0^2. 
\end{equation}
Therefore, observing that
\begin{equation}
    r_*r_\infty=-2\frac
    {\left(1-4m_0^2\right)\left(1+2m_0^2\right)}
    {\left(m_0^2+2\right)^2},
\end{equation}
we can see that
\begin{align}
    C
    &=
    \frac{1-4m_0^2}{r_* r_\infty}\\
    &=
    -\frac{1}{2}\frac{\left(m_0^2+2\right)^2}
    {\left(1+2m_0^2\right)}.
\end{align}
This means that
\begin{equation}
    g(r)=
    -\frac{1}{2}\frac{\left(m_0^2+2\right)^2}
    {\left(1+2m_0^2\right)}
    (r-r_*)(r-r_\infty).
\end{equation}
We can see, therefore, that for all $r_0<r(t)<r_\infty$,
\begin{equation}
    \partial_t (r-r_\infty) \leq 
     -\frac{1}{6}\frac{\left(m_0^2+2\right)^2}
    {\left(1+2m_0^2\right)}
    (r_\infty+2)
    (r_\infty-r_*)\lambda(r-r_\infty).
\end{equation}
Integrating this differential inequality we find that
\begin{equation}
    r(t)-r_\infty 
    \leq
    (r_0-r_\infty)
    \exp\left( -\frac{1}{6}\frac{\left(m_0^2+2\right)^2}
    {\left(1+2m_0^2\right)}
    (r_\infty+2)
    (r_\infty-r_*)
    \int_0^t\lambda(\tau) \diff\tau \right).
\end{equation}
Rearranging this implies that
\begin{equation}
    r(t)\leq r_\infty-
    (r_\infty-r_0)
        \exp\left( -\frac{1}{6}\frac{\left(m_0^2+2\right)^2}
    {\left(1+2m_0^2\right)}
    (r_\infty+2)
    (r_\infty-r_*)
    \int_0^t\lambda(\tau) \diff\tau \right).
\end{equation}
We have already shown that if $r_*<r(t)<r_\infty,$
then $r'(t)>0,$
so we can see that for all $0<t<T_{max}, r_0<r(t)<r_\infty$ and $r$ is increasing.

We will now show that $T_{max}<+\infty$. We will do this by showing that there exists a $0<T<T_{max},$
such that $f\left(r(T)\right)>0$.
We know that $f\left(-\frac{1}{2}\right)=-\frac{3}{2}$
and that $f(r_\infty)=\frac{1}{3}r_\infty$.
We know that $f$ is increasing, so this implies that
there exists a unique $-\frac{1}{2}<\Tilde{r}<r_\infty,$
such that $f(\Tilde{r})=0$.

Suppose towards contradiction that $T_{max}=+\infty$.
Then clearly for all $0<t<+\infty$,
\begin{equation}
    r(t)\leq \Tilde{r}.
\end{equation}
Otherwise there exists $0<t<+\infty$, such that
$r(t)>\Tilde{r}$. We know that $r$ is an increasing function of time, and $f$ is an increasing function for $-\frac{1}{2}<r<1$, so we can see that for all $t>T$,
\begin{equation}
    f(r(t))>f(\Tilde{r})>0.
\end{equation}
This would imply that for all $t>T,$
\begin{equation}
    \partial_t \lambda(t) \geq
    \frac{1}{3}f(\Tilde{r}) \lambda^2.
\end{equation}
This clearly contradicts $T_{max}=+\infty$,
so the hypothesis $T_{max}=+\infty$ clearly implies that for all $0<t<+\infty, r(t)\leq \Tilde{r}$.
Let 
\begin{equation}
    \epsilon=\frac{1}{3}(\Tilde{r}+2)
    \min(g(r_0),g(\Tilde{r})>0.
\end{equation}
Then we have that for all $0<t<+\infty$,
\begin{equation}
    \partial_t r\geq \epsilon \lambda.
\end{equation}
Integrating this differential inequality, we find that
for all $0<t<+\infty$,
\begin{align}
    r(t)&\geq r_0+
    \epsilon\int_0^t\lambda(\tau)\diff\tau \\
    &\geq 
    r_0+ 2\epsilon
    \log\left(\frac{2}{\lambda_0}+t\right).
\end{align}
This clearly contradicts the claim that for all $0<t<+\infty, r(t)<\Tilde{r}$, so we can conclude
that $T_{max}<+\infty$.

It now remains to show that
\begin{align}
    \lim_{t\to T_{max}} r(t)&=r_\infty \\
    \lim_{t\to T_{max}} k(t)&=k_\infty.
\end{align}
We know that $r$ is an increasing function and for all $0<t<T_{max}, r(t)<r_\infty$, so it is clearly the first limit exists and must be less than or equal to $r_\infty$,
so let
\begin{equation}
    \lim_{t\to T_{max}} r(t)=L\leq r_\infty.
\end{equation}
We will now show that $L=r_\infty$.

We will begin by letting
\begin{equation}
    b(t)=\frac{1}{3}\left(-1+2r(t)+2r(t)^2
    +m_0^2(r(t)+2)^2\right).
\end{equation}
Then we have that
\begin{equation}
    \partial_t\lambda=b(t)\lambda^2.
\end{equation}
Integrating this differential equation, we find that for all $0<t_1<t_2<T_{max}$,
\begin{equation}
\frac{1}{\lambda(t_1)}-\frac{1}{\lambda(t_2)}
=\int_{t_1}^{t_2}b(\tau)\diff\tau. 
\end{equation}
Taking the limit $t_2 \to T_{max}$, and using the fact that $\lim_{t\to T_{max}}\lambda(t)=+\infty$,
we can see that for all $0<t<T_{max},$
\begin{equation}
    \frac{1}{\lambda(t)}
    =\int_t^{T_{max}} b(\tau)\diff\tau. 
\end{equation}
For all $0\leq t \leq T_{max},$ let
\begin{equation}
    B(t)=\int_t^{T_{max}} b(\tau)\diff\tau.
\end{equation}
We can immediately see that
$B(T_{max})=0$ and that
\begin{align}
    \partial_t B(T_{max})
    &=
    -\lim_{t\to T_{max}}b(t)\\
    &=
    -\frac{1}{3}\left(-1+2L+2L^2+m_0^2(L+2)^2\right) \\
    &<0.
\end{align}
Applying Taylor's theorem, we can see that for all $0<c<\frac{1}{3}\left(-1+2L+2L^2+m_0^2(L+2)^2\right)$,
there exists $\delta>0,$ such that for all $T_{max}-\delta<t<T_{max}$,
\begin{equation}
    B(t)>c(T_{max}-t).
\end{equation}
This in turn implies that
\begin{equation}
    \lambda(t)>\frac{1}{c(T_{max}-t)}.
\end{equation}

Integrating this bound implies that
for all $T_{max}-\delta<t<T_{max}$,
\begin{align}
    \int_{T_{max}-\delta}^t \lambda(\tau)
    \diff\tau
    &>
    \frac{1}{c}\int_{T_{max}-\delta}^t
    \frac{1}{(T_{max}-\tau)} \diff\tau \\
    &=
    \frac{1}{c}\log\left(
    \frac{T_{max}-\delta}{T_{max}-t}\right).
\end{align}
Recall that for all $0<t<T_{max},$
\begin{equation}
    \partial_t (r-r_\infty)
    =
    -\frac{1}{6}\frac{\left(m_0^2+2\right)^2}
    {\left(1+2m_0^2\right)}\lambda(r+2)
    (r-r_*)\lambda (r-r_\infty)
\end{equation}
Letting
\begin{equation}
    d=-\frac{1}{6}\frac{\left(m_0^2+2\right)^2}
    {\left(1+2m_0^2\right)}(r(T_{max}-\delta)+2)
    (r(T_{max}-\delta)-r_*),
\end{equation}
we can see that
for all $T_{max}-\delta<t<T_{max}$
\begin{equation}
    \partial_t(r-r_{\infty})\geq
    -d\lambda (r-r_\infty)
\end{equation}
Therefore, we can integrate this differential inequality, and find that for all $T_{max}-\delta<t<T_{max}$,
\begin{align}
    r(t)-r_{\infty}
    &\geq 
    \left(r(T_{max}-\delta)-r_\infty\right)
    \exp\left(-d \int_{T_{max}-\delta}^t
    \lambda(\tau)\diff\tau \right) \\
    &>
    \left(r(T_{max}-\delta)-r_\infty\right)
    \exp\left(
    -\frac{d}{c}\log\left(
    \frac{T_{max}-\delta}{T_{max}-t}\right)
    \right) \\
    &=
    \left(r(T_{max}-\delta)-r_\infty\right)
    \left(
    \frac{T_{max}-\delta}{T_{max}-t}
    \right)^{-\frac{d}{c}}.
\end{align}
This means that
for all $T_{max}-\delta<t<T_{max}$,
\begin{equation}
    r_\infty-
    \left(r_\infty- r(T_{max}-\delta)\right)
    \left(
    \frac{T_{max}-t}{T_{max}-\delta}
    \right)^\frac{d}{c}
    <r(t)<r_\infty,
\end{equation}
and so we may conclude that
\begin{equation}
    \lim_{t\to T_{max}}r(t)=r_\infty.
\end{equation}

Finally, we must show that
\begin{equation}
    \lim_{t\to T_{max}}k(t)=k_\infty,
\end{equation}
Using the fact that $k(t)=m_0(r(t)+2),$
we can see that
\begin{align}
    \lim_{t\to T_{max}}k(t)
    &=
    m_0(r_\infty+2) \\
    &=
     m_0 \left(
    \frac{9+3\sqrt{1-4m_0^2}}{2m_0^2+4}\right) \\
    &=
    k_\infty.
\end{align}
This completes the proof.
\end{proof}

\begin{theorem}
Suppose we have initial data
\begin{equation}
    S^0=
    \left(
    \begin{array}{ccc}
         -(r_0+1)\lambda_0& 0 & 0  \\
          0& r_0 \lambda_0 &0 \\
          0 &0 &\lambda_0 
    \end{array}
    \right)
\end{equation}
and 
\begin{equation}
    \omega^0= 
    \left(
    \begin{array}{c}
         0  \\
         2k_0\lambda_0 \\
         0
    \end{array}
    \right),
\end{equation}
where $\lambda_0>0, r_0>0$,
and
$1+r_0-2r_0^2<k_0^2<\frac{1}{4}(r_0+2)^2$,
and suppose the pair 
$S,\omega\in C^1\left([0,T_{max})\right)$
is a solution of the strain-vorticity pair equation where
the strain is given by
\begin{equation}
    S(t)=
    \left(
    \begin{array}{ccc}
         -(r(t)+1)\lambda(t)& 0 & 0  \\
          0& r(t) \lambda(t) &0 \\
          0 &0 &\lambda(t) 
    \end{array}
    \right),
\end{equation}
and a vorticity given by
\begin{equation}
    \omega(t)= 
    \left(
    \begin{array}{c}
         0  \\
         2 k(t) \lambda(t) \\
         0
    \end{array}
    \right).
\end{equation}
Then $T_{max}<+\infty$ and 
\begin{align*}
    \lim_{t\to T_{max}} r(t)&=r_\infty, \\
    \lim_{t\to T_{max}}k(t)&=k_\infty, \\
    \lim_{t\to T_{max}}g(t)&=0, \\
    \lim_{t\to T_{max}}\lambda(t)&=+\infty,
\end{align*}
where
\begin{align}
    r_\infty&=\frac{1-4m_0^2
    +3\sqrt{1-4m_0^2}}{2m_0^2+4},
    \\
    k_\infty&= m_0 \left(
    \frac{9+3\sqrt{1-4m_0^2}}{2m_0^2+4}\right).
\end{align}
and
\begin{equation}
    m_0=\frac{k_0}{r_0+2}
\end{equation}
\end{theorem}

\begin{proof}
We will proceed in a similar manner as in the proof of Theorem \ref{BlowupFromBelow}.
We will begin by recalling that
\begin{equation}
    \partial_t r(t)
    =-\frac{1}{6}\frac{\left(m_0^2+2\right)^2}{1+m_0^2}
    \lambda (r+2)(r-r_*)(r-r_\infty).
\end{equation}
Using this differential equality, we will show that for all 
$0<t<T_{max},$
we will have
\begin{equation}
    r_\infty<r(t)<r_0,
\end{equation}
and that $r$ is a decreasing function for $0<t<T_{max}$.
We can easily see that for all $r(t)>r_\infty$,
\begin{equation}
    \partial_t r(t)<0.
\end{equation}
This implies that for all $0<t<T_{max}, r(t)<r_0$.
This in turn implies that for all $0<t<T_{max}$
\begin{equation}
    \partial_t(r-r_\infty)
    >
    -\frac{1}{6}\frac{\left(m_0^2+2\right)^2}{1+m_0^2}
    \lambda (r_0+2)(r_0-r_*)(r-r_\infty).
\end{equation}
Integrating this differential inequality, we find that
for all $0<t<T_{max}$,
\begin{equation}
    r(t)\geq r_\infty+(r_0-r_\infty)
    \exp\left(-\frac{1}{6}\frac{\left(m_0^2+2\right)^2}
    {1+m_0^2}\int_0^t \lambda(\tau)\diff\tau \right),
\end{equation}
and so for all $0<t<T_{max}, r(t)>r_\infty$.
We have already shown that if $r(t)>r_\infty$, then
\begin{equation}
    \partial_t r(t)<0.
\end{equation}
so now we can conclude that $r$ is a decreasing function for $0<t<T_{max}$, and we have now proven our statement about the range of $r$.

Next we will note that, as we have already shown in the proof of the above theorem, $f$ is an increasing function.
We will also compute that 
\begin{equation}
    f(r_\infty)=\frac{1}{3}r_\infty.
\end{equation}
This implies that for all $0<t<T_{max}$,
\begin{equation}
    \partial_t \lambda >r_\infty \lambda^2.
\end{equation}
Integrating this differential inequality we find that for all $0<t<T_{max}$,
\begin{equation}
    \lambda(t)>\frac{\lambda_0}{1-r_\infty\lambda_0 t}.
\end{equation}
This implies we have finite-time blowup with
$T_{max}<\frac{1}{r_\infty \lambda_0}$.

It now remains to show that 
\begin{equation}
    \lim_{t\to T_{max}}r(t)=r_\infty.
\end{equation}
The proof will essentially be the same in the above theorem.
We know that $r$ is a decreasing function and for all $0<t<T_{max}, r(t)>r_\infty$, so it is clearly the first limit exists and must be greater than or equal to $r_\infty$,
so let
\begin{equation}
    \lim_{t\to T_{max}} r(t)=L\geq r_\infty.
\end{equation}
We will now show that $L=r_\infty$.

We will begin by letting
\begin{equation}
    b(t)=\frac{1}{3}\left(-1+2r(t)+2r(t)^2
    +m_0^2(r(t)+2)^2\right).
\end{equation}
Then we have that
\begin{equation}
    \partial_t\lambda=b(t)\lambda^2.
\end{equation}
Integrating this differential equation, we find that for all $0<t_1<t_2<T_{max}$,
\begin{equation}
\frac{1}{\lambda(t_1)}-\frac{1}{\lambda(t_2)}
=\int_{t_1}^{t_2}b(\tau)\diff\tau. 
\end{equation}
Taking the limit $t_2 \to T_{max}$, and using the fact that $\lim_{t\to T_{max}}\lambda(t)=+\infty$,
we can see that for all $0<t<T_{max},$
\begin{equation}
    \frac{1}{\lambda(t)}
    =\int_t^{T_{max}} b(\tau)\diff\tau. 
\end{equation}
For all $0\leq t \leq T_{max},$ let
\begin{equation}
    B(t)=\int_t^{T_{max}} b(\tau)\diff\tau.
\end{equation}
We can immediately see that
$B(T_{max})=0$ and that
\begin{align}
    \partial_t B(T_{max})
    &=
    -\lim_{t\to T_{max}}b(t)\\
    &=
    -\frac{1}{3}\left(-1+2L+2L^2+m_0^2(L+2)^2\right) \\
    &<0.
\end{align}
Applying Taylor's theorem, we can see that for all $0<c<\frac{1}{3}\left(-1+2L+2L^2+m_0^2(L+2)^2\right)$,
there exists $\delta>0,$ such that for all $T_{max}-\delta<t<T_{max}$,
\begin{equation}
    B(t)>c(T_{max}-t).
\end{equation}
This in turn implies that
\begin{equation}
    \lambda(t)>\frac{1}{c(T_{max}-t)}.
\end{equation}

Integrating this bound implies that
for all $T_{max}-\delta<t<T_{max}$,
\begin{align}
    \int_{T_{max}-\delta}^t \lambda(\tau)
    \diff\tau
    &>
    \frac{1}{c}\int_{T_{max}-\delta}^t
    \frac{1}{(T_{max}-\tau)} \diff\tau \\
    &=
    \frac{1}{c}\log\left(
    \frac{T_{max}-\delta}{T_{max}-t}\right).
\end{align}
Recall that for all $0<t<T_{max},$
\begin{equation}
    \partial_t (r-r_\infty)
    =
    -\frac{1}{6}\frac{\left(m_0^2+2\right)^2}
    {\left(1+2m_0^2\right)}\lambda (r+2)
    (r-r_*)\lambda (r-r_\infty)
\end{equation}
Letting
\begin{equation}
    d=-\frac{1}{6}\frac{\left(m_0^2+2\right)^2}
    {\left(1+2m_0^2\right)}(r_\infty+2)
    (r_\infty-r_*),
\end{equation}
we can see that
for all $T_{max}-\delta<t<T_{max}$
\begin{equation}
    \partial_t(r-r_{\infty})\leq
    -d\lambda (r-r_\infty)
\end{equation}
Therefore, we can integrate this differential inequality, and find that for all $T_{max}-\delta<t<T_{max}$,
\begin{align}
    r(t)-r_{\infty}
    &\leq 
    \left(r(T_{max}-\delta)-r_\infty\right)
    \exp\left(-d \int_{T_{max}-\delta}^t
    \lambda(\tau)\diff\tau \right) \\
    &<
    \left(r(T_{max}-\delta)-r_\infty\right)
    \exp\left(
    -\frac{d}{c}\log\left(
    \frac{T_{max}-\delta}{T_{max}-t}\right)
    \right) \\
    &=
    \left(r(T_{max}-\delta)-r_\infty\right)
    \left(
    \frac{T_{max}-\delta}{T_{max}-t}
    \right)^{-\frac{d}{c}}.
\end{align}
This means that
for all $T_{max}-\delta<t<T_{max}$,
\begin{equation}
    r_\infty<r(t)<r_\infty+
    \left(r(T_{max}-\delta)-r_\infty\right)
    \left(
    \frac{T_{max}-t}{T_{max}-\delta}
    \right)^\frac{d}{c},
\end{equation}
and so we may conclude that
\begin{equation}
    \lim_{t\to T_{max}}r(t)=r_\infty.
\end{equation}
Finally we observe that
\begin{equation}
    \lim_{t \to T_{max}}k(t)=
    m_0(r_\infty+2)=k_\infty.
\end{equation}
This completes the proof.
\end{proof}

\begin{theorem} \label{BoundaryCase}
Suppose we have initial data
\begin{equation}
    S^0=
    \left(
    \begin{array}{ccc}
         -(r_0+1)\lambda_0& 0 & 0  \\
          0& r_0 \lambda_0 &0 \\
          0 &0 &\lambda_0 
    \end{array}
    \right)
\end{equation}
and 
\begin{equation}
    \omega^0= 
    \left(
    \begin{array}{c}
         0  \\
         2k_0\lambda_0 \\
         0
    \end{array}
    \right),
\end{equation}
where $\lambda_0>0, r_0>0$, and 
$k_0^2=\frac{1}{4}(r_0+2)^2$,
and suppose the pair 
$S,\omega\in C^1\left([0,T_{max})\right)$
is a solution of the strain-vorticity pair equation where
the strain is given by
\begin{equation}
    S(t)=
    \left(
    \begin{array}{ccc}
         -(r(t)+1)\lambda(t)& 0 & 0  \\
          0& r(t) \lambda(t) &0 \\
          0 &0 &\lambda(t) 
    \end{array}
    \right),
\end{equation}
and a vorticity given by
\begin{equation}
    \omega(t)= 
    \left(
    \begin{array}{c}
         0  \\
         2 k(t) \lambda(t) \\
         0
    \end{array}
    \right),
\end{equation}
Then $T_{max}=+\infty$, and
\begin{align*}
    \lim_{t\to +\infty} r(t)&=0, \\
    \lim_{t\to +\infty}k(t)&=1, \\
    \lim_{t\to +\infty}g(t)&=0.
\end{align*}
There is, however, blowup at infinity with
\begin{equation*}
    \lim_{t\to +\infty}\lambda(t)=+\infty.
\end{equation*}
\end{theorem}

Before we prove this theorem, we will first need to prove a lemma that gives us a phase space analysis of the behaviour of the variables $\lambda,r$ in the case were 
$m_0^2=\frac{1}{4}$.

\begin{lemma} \label{BoundaryCaseLemma}
Suppose $\lambda,r,k\in C^1\left([0,T_{max})\right)$ 
is a solution of the strain-vorticity parameter system of ODEs,
where $r_0>0$ and
\begin{equation}
    k_0^2=\frac{1}{4}(r_0+2)^2
\end{equation}
Then for all $0<t<T_{max}$
\begin{equation}
    \lambda(t)=c_0 r(t)^{-\frac{1}{2}}(r(t)+2)^{-\frac{5}{6}},
\end{equation}
where
\begin{equation}
    c_0=\lambda_0 r_0^{\frac{1}{2}}(r_0+2)^{\frac{5}{6}}
\end{equation}
\end{lemma}

\begin{proof}
We will start by letting 
$c(t)=\lambda(t) r(t)^{\frac{1}{2}}(r(t)+2)^{\frac{5}{6}}$.
Recalling Corollary \ref{SimplifiedSystem}, we can see that
\begin{align}
    \partial_t\lambda &=
    \frac{1}{3}\left(-1+2r+2r^2+\frac{1}{4}(r+2)^2\right)
    \lambda^2, \\
    \partial_t r &=
    \frac{1}{3}\lambda(r+2)
    \left(1+r-2r^2-\frac{1}{4}(r+2)^2\right),
\end{align}
and that the equations further simplify to
\begin{align}
    \partial_t\lambda &=
    r\left(r+\frac{3}{4}\right)
    \lambda^2, \\
    \partial_t r &=
    -\frac{3}{4}\lambda(r+2)r^2.
\end{align}
Observe that
\begin{equation}
    c(t)^6=\lambda(t)^6r(t)^3(r(t)+2)^5,
\end{equation}
and differentiate to find that
\begin{align}
    6c(t)^5\partial_t c(t)
    &=
    6\lambda^5r^3(r+2)^5 \lambda'
    +3\lambda^6r^2(r+2)^5 r'
    +5\lambda^5r^3(r+2)^4 r' \\
    &=
    \lambda^5r^2(r+2)^4\left(6r(r+2)\lambda'
    +(8\lambda r+6\lambda)r'\right) \\
    &=
    \lambda^5r^2(r+2)^4\left(6r(r+2)
    r\left(r+\frac{3}{4}\right)
    \lambda^2
    -\frac{3}{4}(8\lambda r+6\lambda)
    \lambda(r+2)r^2\right) \\
    &=
    \lambda^7r^4(r+2)^4\left(6(r+2)\left(r+\frac{3}{4}\right)
    -\frac{3}{4}(r+2)(8r+6)\right) \\
    &=0.
\end{align}
Therefore we may conclude that for all $c(t)>0, \partial_tc(t)=0.$
Recalling that by hypothesis $r_0,\lambda_0>0,$ we can see that $c_0>0,$ and therefore that for all $0<t<T_{max}$,
\begin{equation}
    c(t)=c_0.
\end{equation}
This implies that for all $0<t<T_{max}$,
\begin{equation}
    \lambda(t) r(t)^{\frac{1}{2}}(r(t)+2)^{\frac{5}{6}}=c_0,
\end{equation}
and this completes the proof.
\end{proof}

Having proven this lemma, we can now prove Theorem \ref{BoundaryCase}.

\begin{proof}
We will begin by fixing
\begin{equation}
    c_0=\lambda_0 r_0^{\frac{1}{2}}(r_0+2)^{\frac{5}{6}}.
\end{equation}
Applying Lemma \ref{BoundaryCaseLemma}, we can see that
for all $0<t<T_{max}$,
\begin{equation}
    \lambda(t)=c_0r(t)^{-\frac{1}{2}}
    (r(t)+2)^{-\frac{5}{6}},
\end{equation}
and that therefore
\begin{equation}
    \partial_t r=-\frac{3}{4}c_0 
    (r+2)^\frac{1}{6}r^\frac{3}{2}.
\end{equation}

First we will show that for all $0<t<T_{max}, 0<r(t)<r_0$.
We will begin by computing that
\begin{equation}
    \partial_t r(t)^-\frac{1}{2}=
    \frac{3}{8}(r(t)+2)^\frac{1}{6}.
\end{equation}
Integrating this differential equality we find that for all
$0<t<T_{max}$
\begin{equation}
   r(t)^{-\frac{1}{2}}-r_0^{-\frac{1}{2}}
   =\frac{3}{8}\int_0^t(r(t)+2)^\frac{1}{6}\diff\tau,
\end{equation}
and so we find that
\begin{equation} \label{identity}
    r(t)=\frac{r_0}{\left(1+\frac{3}{8}r_0^\frac{1}{2}
    \int_0^t(r(\tau)+2)^\frac{1}{6}\diff\tau\right)^2}.
\end{equation}
Suppose towards contradiction that there exists 
$0<t<T_{max}$, such that $r(t)=0$. Let $T$ be the first such time. This means that $r(T)=0,$ and for all 
$0<t<T, r(t)>0$.
Therefore for all $0<t<T, r(t)+2>2,$ and so
\begin{equation}
    r(T)>\frac{r_0}{\left(1+\frac{3}{8}r_0^\frac{1}{2}
    2^\frac{1}{6}T\right)^2}.
\end{equation}
This clearly contradicts the assumption that $r(T)=0$,
and therefore we may conclude that for all $0<t<T_{max}, r(t)>0.$
Likewise, we can immediately see that because $r(t)>0$,
we also have for all $0<t<T_{max}, \partial_tr(t)<0$, and that therefore $r$ is a decreasing function.

We have now proven that for all $0<t<T_{max}$,
\begin{equation}
    0<r(t)<r_0.
\end{equation}
It then immediately follows that for all
$0<t<T_{max}$,
\begin{equation}
    2^\frac{1}{6}t
    <\int_0^t(r(\tau)+2)^\frac{1}{6}\diff\tau
    <(r_0+2)^\frac{1}{6}t.
\end{equation}
Applying these bounds to the equation \eqref{identity},
we find that for all $0<t<T_{max}$,
\begin{equation} \label{rBound}
    \frac{r_0}{\left(1+\frac{3}{8}r_0^\frac{1}{2}
    2^\frac{1}{6}t\right)^2}
    <r(t)
    <\frac{r_0}{\left(1+\frac{3}{8}r_0^\frac{1}{2}
    (r_0+2)^\frac{1}{6}t\right)^2}
\end{equation}
Recalling from Lemma \ref{BoundaryCaseLemma}, we know that
\begin{equation}
    \lambda(t)=c_0r(t)^{-\frac{1}{2}}
    (r(t)+2)^{-\frac{5}{6}},
\end{equation}
and so therefore
\begin{align}
    \lambda(t)
    &> \label{LowerBound}
    c_0 r_0^{-\frac{1}{2}}
    \left(1+\frac{3}{8}r_0^\frac{1}{2}
    (r_0+2)^\frac{1}{6}t\right)
    \left(\frac{r_0}{\left(1+\frac{3}{8}r_0^\frac{1}{2}
    (r_0+2)^\frac{1}{6}t\right)^2}
    +2\right)^{-\frac{5}{6}}\\
    \lambda(t)
    &< \label{UpperBound}
    c_0 r_0^{-\frac{1}{2}}
    \left(1+\frac{3}{8}r_0^\frac{1}{2}
    2^\frac{1}{6}t\right)
    \left(\frac{r_0}{\left(1+\frac{3}{8}r_0^\frac{1}{2}
    2^\frac{1}{6}t\right)^2}
    +2\right)^{-\frac{5}{6}}.
\end{align}
We can see that $r$ is bounded, so in order for there to be finite-time blowup at $T_{max}<+\infty$, we must have
\begin{equation}
    \limsup_{t\to T_{max}} \lambda(t)=+\infty,
\end{equation}
and therefore the bound \eqref{UpperBound} immediately shows that $T_{max}=+\infty$.
Likewise the bounds \eqref{rBound} and \eqref{LowerBound} respectively show that
\begin{align}
    \lim_{t\to+\infty} r(t) &= 0 \\
    \lim_{t\to+\infty} \lambda(t) &= +\infty.
\end{align}
This completes the proof.
\end{proof}

\begin{theorem}. \label{overshoot}
Suppose we have initial data
\begin{equation}
    S^0=
    \left(
    \begin{array}{ccc}
         -(r_0+1)\lambda_0& 0 & 0  \\
          0& r_0 \lambda_0 &0 \\
          0 &0 &\lambda_0 
    \end{array}
    \right)
\end{equation}
and 
\begin{equation}
    \omega^0= 
    \left(
    \begin{array}{c}
         0  \\
         2k_0\lambda_0 \\
         0
    \end{array}
    \right),
\end{equation}
where $\lambda_0>0, r_0>0, k_0^2>\frac{1}{4}(r_0+2)^2$,
and suppose the pair 
$S,\omega\in C^1\left([0,T_{max})\right)$
is a solution of the strain-vorticity pair equation where
the strain is given by
\begin{equation}
    S(t)=
    \left(
    \begin{array}{ccc}
         -(r(t)+1)\lambda(t)& 0 & 0  \\
          0& r(t) \lambda(t) &0 \\
          0 &0 &\lambda(t) 
    \end{array}
    \right),
\end{equation}
and a vorticity given by
\begin{equation}
    \omega(t)= 
    \left(
    \begin{array}{c}
         0  \\
         2 k(t) \lambda(t) \\
         0
    \end{array}
    \right).
\end{equation}
Then $T_{max}<+\infty$ and 
\begin{align*}
    \lim_{t\to T_{max}} r(t)&=-2, \\
    \lim_{t\to T_{max}}k(t)&=0, \\
    \lim_{t\to T_{max}}\lambda(t)&=+\infty.
\end{align*}
\end{theorem}

\begin{proof}
First we will observe that
$\frac{1}{4}(r+2)^2=1+r+\frac{1}{4}r^2$ and that
$m_0^2>\frac{1}{4}$, and so we have
\begin{equation}
    \partial_t r=\left(\left(-(m_0^2-\frac{1}{4})(r+2)^2\right)
    -\frac{9}{4}r^2\right)\lambda(r+2).
\end{equation}
We can clearly see that if $r(t)>-2$, then
$\partial_t r\leq 0$. Letting
\begin{equation}
    \sigma=
    \left(m_0^2-\frac{1}{4}\right)(r_0+2)^2
    +\max \left(9, \frac{9}{4}r_0^2\right),
\end{equation}
we can see that for all $-2<r(t)<r_0,$
\begin{equation}
    \partial_t r(t)\geq -\sigma \lambda (r+2),
\end{equation}
and that likewise
\begin{equation}
    \partial_t (r(t)+2)\geq -\sigma \lambda (r+2).
\end{equation}
Integrating this differential inequality, we find that for all $0<t<T_{max},$
\begin{equation}
    r(t)\geq -2+(r_0+2)\exp\left(-\sigma
    \int_0^t\lambda(\tau)\diff\tau\right).
\end{equation}
This implies that for all $0<t<T_{max}$, $-2<r(t)<r_0$, and that $r$ is a decreasing function on the time interval $0<t<T_{max}$.

Next we will observe that for all $0<t<T_{max}$,
\begin{equation}
    \partial_t r\leq -\left(m_0^2-\frac{1}{4}\right)
    \lambda (r+2)^3.
\end{equation}
Note that this also means that
\begin{equation} 
    \partial_t (r+2)\leq -\left(m_0^2-\frac{1}{4}\right)
    \lambda (r+2)^3.
\end{equation}
Integrating this differential inequality we find that
for all $0<t<T_{max}$,
\begin{equation} \label{UpperBoundOvershoot}
    r(t)\leq -2+\frac{r_0+2}{\sqrt{1+2(r_0+2)^2
    \left(m_0^2-\frac{1}{4}\right)
    \int_0^t \lambda(\tau)\diff\tau}}.
\end{equation}

Next we will show that $T_{max}<+\infty$.
Recall that
\begin{equation}
    \partial_t\lambda =
    \frac{1}{3}\left(-1+2r+2r^2+m_0^2(r+2)^2\right)
    \lambda^2,
\end{equation}
and define $f$ as in Theorem \ref{BlowupFromBelow}
\begin{equation}
    f(r)=-1+2r+2r^2+m_0^2(r+2)^2.
\end{equation}
We will define
\begin{equation}
    \rho= \inf_{r\in\mathbb{R}}
    \frac{1}{3}\left(-1+2r+2r^2+m_0^2(r+2)^2\right),
\end{equation}
and observe that
\begin{equation}
    \partial_t\lambda \geq \rho \lambda^2.
\end{equation}
If $\rho>0$, then we can see that for all $0<t<T_{max}$
\begin{equation}
    \lambda(t)\geq \frac{\lambda_0}
    {1-\rho\lambda_0 t}.
\end{equation}
This implies that
\begin{equation}
    T_{max}\leq \frac{1}{\rho\lambda_0}<+\infty,
\end{equation}
and so clearly there is finite time blowup.

Now we will consider the case $\rho<0$.
Let $r'$ be the smaller zero of $f$.
This implies that for all $r<r', f(r)>0$.
Suppose towards contradiction that $T_{max}<+\infty$.
This immediately implies that for all $0<t<+\infty,$
\begin{equation}
    r(t)\geq r'.
\end{equation}
Otherwise, there exists $T>0$, such that $r(T)<r'$, and because $r$ is a decreasing function for all $0<t<T_{max}$,
and $f$ is a decreasing function for $-\infty<r<r'$,
we can see that for all $t>T$,
\begin{equation}
    f(r(t))>f(r(T)).
\end{equation}
Therefore we would have for all $t>T$,
\begin{equation}
    \partial_t \lambda \geq \frac{1}{3}f(r(T))\lambda^2,
\end{equation}
with $f(r(T))>0$.
This clearly contradicts $T_{max}=+\infty$, so we can conclude that for all $0<t<T_{max}, r(t)\geq r'$.

Now let $\gamma=-\rho$.
We then have that for all 
$0<t<+\infty, \partial_t\lambda\geq -\gamma\lambda^2$,
and integrating this differential inequality, we find that for all $0<t<T_{max}$,
\begin{equation}
    \lambda(t)\geq 
    \frac{\lambda_0}{1+\gamma\lambda_0 t}.
\end{equation}
Integrating this inequality we find that for all
$0<t<+\infty,$
\begin{equation}
    \int_0^t \lambda(\tau)\diff\tau \geq
    \frac{1}{\gamma}\log 
    \left(1+ \gamma\lambda_0 t\right).
\end{equation}
Applying this inequality to \eqref{UpperBoundOvershoot},
we find that
\begin{equation}
        r(t)\leq -2+\frac{r_0+2}{\sqrt{1+2(r_0+2)^2
    \left(m_0^2-\frac{1}{4}\right)
    \frac{1}{\gamma}\log 
    \left(1+ \gamma\lambda_0 t\right)}}.
\end{equation}
Taking the limit $t\to +\infty$, we can see that
\begin{equation}
    \lim_{t\to+\infty} r(t)=-2.
\end{equation}
It is easy to check that $r'>-2$, so this contradicts the assertion that for all $0<t<+\infty, r(t)\geq r'$, and so we can conclude that $T_{max}<+\infty$.

Now we will consider the case $\rho=0$. Let $r'$ be the only zero of $f$ (with multiplicity two). Clearly we have for all $r<r', f(r)>0$.
Suppose towards contradiction that $T_{max}<+\infty$.
This immediately implies that for all $0<t<+\infty,$
\begin{equation}
    r(t)\geq r'.
\end{equation}
Otherwise, there exists $T>0$, such that $r(T)<r'$, and because $r$ is a decreasing function for all $0<t<T_{max}$,
and $f$ is a decreasing function for $-\infty<r<r'$,
we can see that for all $t>T$,
\begin{equation}
    f(r(t))>f(r(T)).
\end{equation}
Therefore we would have for all $t>T$,
\begin{equation}
    \partial_t \lambda \geq \frac{1}{3}f(r(T))\lambda^2,
\end{equation}
with $f(r(T))>0$.
This clearly contradicts $T_{max}=+\infty$, so we can conclude that for all $0<t<T_{max}, r(t)\geq r'$.

We know that for all 
$0<t<+\infty, \partial_t\lambda\geq 0,$
so we can see that for all $0<t<+\infty, \lambda(t)\geq \lambda^0$. Plugging into the bound \eqref{UpperBoundOvershoot}, we find that
\begin{align}
    r(t)
    &\leq 
    -2+\frac{r_0+2}{\sqrt{1+2(r_0+2)^2
    \left(m_0^2-\frac{1}{4}\right)
    \int_0^t \lambda(\tau)\diff\tau}} \\
    &\leq 
    -2+\frac{r_0+2}{\sqrt{1+2(r_0+2)^2
    \left(m_0^2-\frac{1}{4}\right)
    \lambda_0 t}}.
\end{align}
Taking the limit $t\to +\infty$, we can see that
\begin{equation}
    \lim_{t\to+\infty} r(t)=-2.
\end{equation}
It is easy to check that $r'>-2$, so this contradicts the assertion that for all $0<t<+\infty, r(t)\geq r'$, and so we can conclude that $T_{max}<+\infty$.

It now remains to show that
\begin{align*}
    \lim_{t\to T_{max}} r(t)&=-2, \\
    \lim_{t\to T_{max}}k(t)&=0.
\end{align*}
We know that $r$ is an decreasing function and for all $0<t<T_{max}, r(t)>-2$, so it is clearly the first limit exists and must be greater than or equal to $-2$,
so let
\begin{equation}
    \lim_{t\to T_{max}} r(t)=L\geq -2.
\end{equation}
We will now show that $L=-2$.

We will begin by letting
\begin{equation}
    b(t)=\frac{1}{3}\left(-1+2r(t)+2r(t)^2
    +m_0^2(r(t)+2)^2\right).
\end{equation}
Then we have that
\begin{equation}
    \partial_t\lambda=b(t)\lambda^2.
\end{equation}
Integrating this differential equation, we find that for all $0<t_1<t_2<T_{max}$,
\begin{equation}
\frac{1}{\lambda(t_1)}-\frac{1}{\lambda(t_2)}
=\int_{t_1}^{t_2}b(\tau)\diff\tau. 
\end{equation}
Taking the limit $t_2 \to T_{max}$, and using the fact that $\lim_{t\to T_{max}}\lambda(t)=+\infty$,
we can see that for all $0<t<T_{max},$
\begin{equation}
    \frac{1}{\lambda(t)}
    =\int_t^{T_{max}} b(\tau)\diff\tau. 
\end{equation}
For all $0\leq t \leq T_{max},$ let
\begin{equation}
    B(t)=\int_t^{T_{max}} b(\tau)\diff\tau.
\end{equation}
We can immediately see that
$B(T_{max})=0$ and that
\begin{align}
    \partial_t B(T_{max})
    &=
    -\lim_{t\to T_{max}}b(t)\\
    &=
    -\frac{1}{3}\left(-1+2L+2L^2+m_0^2(L+2)^2\right) \\
    &<0.
\end{align}
Applying Taylor's theorem, we can see that for all $0<c<\frac{1}{3}\left(-1+2L+2L^2+m_0^2(L+2)^2\right)$,
there exists $\delta>0,$ such that for all $T_{max}-\delta<t<T_{max}$,
\begin{equation}
    B(t)>c(T_{max}-t).
\end{equation}
This in turn implies that
\begin{equation}
    \lambda(t)>\frac{1}{c(T_{max}-t)}.
\end{equation}

Integrating this bound implies that
for all $T_{max}-\delta<t<T_{max}$,
\begin{align}
    \int_0^t \lambda(\tau)
    &>
    \int_{T_{max}-\delta}^t \lambda(\tau) \\
    \diff\tau
    &>
    \frac{1}{c}\int_{T_{max}-\delta}^t
    \frac{1}{(T_{max}-\tau)} \diff\tau \\
    &=
    \frac{1}{c}\log\left(
    \frac{T_{max}-\delta}{T_{max}-t}\right).
\end{align}
Plugging this inequality into \eqref{UpperBoundOvershoot},
we find that for all 
$T_{max}-\delta<t<T_{max}$
\begin{equation}
    -2< r(t)
    <
     -2+\frac{r_0+2}{\sqrt{1+2(r_0+2)^2
    \left(m_0^2-\frac{1}{4}\right)
    \frac{1}{c}\log\left(
    \frac{T_{max}-\delta}{T_{max}-t}\right)}},
\end{equation}
and so taking the limit $t\to T_{max}$, we find that
\begin{equation}
    \lim_{t\to T_{max}} r(t)=-2.
\end{equation}
Finally we will observe that for all $0<t<T_{max}, k(t)=m_0(r(t)+2)$, and so we can immediately conclude that
\begin{equation}
    \lim_{t\to T_{max}} k(t)=0.
\end{equation}
This completes the proof.
\end{proof}

\begin{remark}
We have now proven all the cases of Theorem \ref{MainTheorem} except the cases where 
$-\frac{1}{2}\leq r_0\leq 0, k_0^2>1+r_0-2r_0^2$ 
and where $r_0< -\frac{1}{2}$.
The proofs for these two cases, however, are entirely analogous to the proof of Theorem \ref{overshoot}, and so they are omitted to avoid further cluttering up the paper. The argument that $r=-2$ is an attractor in this region is essentially the same, the only differences are minor technicalities in how the necessary differential inequalities are manipulated.
\end{remark}

\end{document}